\documentclass[11pt,reqno]{amsart}

\usepackage[utf8]{inputenc}
\usepackage{microtype}

\usepackage[T1]{fontenc}
\usepackage{amssymb}
\usepackage{amsthm}
\usepackage{amscd}
\usepackage{amsfonts}
\usepackage{stmaryrd}
\usepackage[all]{xy}

\usepackage{caption}
\usepackage{subcaption}
\usepackage{euler}

\usepackage{extarrows}

\usepackage[colorlinks, linktocpage, citecolor = purple, linkcolor = blue]{hyperref}
\usepackage{color}

\usepackage{tikz}									
\usetikzlibrary{matrix}
\usetikzlibrary{patterns}
\usetikzlibrary{matrix}
\usetikzlibrary{positioning}
\usetikzlibrary{decorations.pathmorphing}
\usetikzlibrary{cd}

\usepackage{fullpage}

\usepackage[shortlabels]{enumitem}

\usepackage{amsmath,xypic}								
\usepackage{amssymb}
\usepackage{amsthm}
\usepackage{amscd}
\usepackage{amsfonts}

\usepackage{caption}
\usepackage{subcaption}

\usepackage{tikz}									
\usetikzlibrary{matrix}
\usetikzlibrary{patterns}
\usetikzlibrary{matrix}
\usetikzlibrary{positioning}
\usetikzlibrary{decorations.pathmorphing}
\usetikzlibrary{cd}

\usepackage[shortlabels]{enumitem}

\newtheorem{theorem}{Theorem}[section]
\newtheorem{lemma}[theorem]{Lemma}

\newtheorem{prop}[theorem]{Proposition}
 
\newtheorem{cor}[theorem]{Corollary} 
 
\usepackage{comment}

\theoremstyle{definition}

\newtheorem{mydef}[theorem]{Definition}

\newtheorem{exa}[theorem]{Example}
\newtheorem{algo}[theorem]{Algorithm}

\newtheorem{rem}[theorem]{Remark}

\author[M. Brandt]{Madeline Brandt}
\address{Department of Mathematics, Brown University,  Providence, RI 02912}
\email{\href{mailto:madeline_brandt@brown.edu}{madeline\_brandt@brown.edu}}

\author{Paul Alexander Helminck}
  \address{Department of Mathematics, Swansea University, Swansea}
 \email{\href{mailto:paulhelminck@gmail.com}{paulhelminck@gmail.com}}

   \title{Tropical superelliptic curves}

\begin{document}

 \begin{abstract}
  We present an algorithm for computing the Berkovich skeleton of a superelliptic curve $y^n=f(x)$ over a valued field. After defining superelliptic weighted metric graphs, we show that each one is realizable by an algebraic superelliptic curve when $n$ is prime. 
Lastly, we study the locus of superelliptic weighted metric graphs inside the moduli space of tropical curves of genus $g$.   
 \end{abstract}

\maketitle

\section{Introduction}
In this paper we study tropical superelliptic curves and tropicalizations of superelliptic covers.
Let $K$ be a field of characteristic 0 that is complete with respect to a nontrivial non-Archimedean discrete valuation. Let $R$ be the valuation ring of $K$ with maximal ideal $m$, let $k := R/m$ be the residue field, and let $\pi$ be a uniformizer for $K$. 
A \emph{superelliptic curve} over $K$ is a curve $X$ which admits a Galois covering $\phi:X\rightarrow{\mathbb{P}^1}$ such that the Galois group is cyclic of order $n$. We assume the characteristic of the residue field $k$ is relatively prime to $n$.
Assuming $K$ contains $n$ distinct primitive $n$-th roots of unity, Kummer theory \cite[Proposition 3.2]{neu} tells us the covering comes as
$
y^n=f(x),
$
where $f(x)$ is some rational function in $K(x)$. This normal form allows us to directly relate ramification data of the corresponding covering $(x,y)\mapsto{x}$ to the rational function $f$.
We can in fact assume $f(x)$ is a polynomial by the following transformation. For $f(x)$ of the form $f(x)=g(x)/h(x)$, we multiply both sides of $y^n=f(x)$ by $h(x)^{n}$ and make a change of coordinates $\tilde{y}=h(x)\cdot{y}$ to obtain the integral equation $(\tilde{y})^n=g(x)h(x)^{n-1}$.

Finding the Berkovich skeleton of a curve $X$ over $K$ given defining equations of $X$ is a difficult problem. There are some theoretical procedures for carrying this out, and these usually involve finding the finite extension of $K$ needed for semistable reduction and calculating a regular model by normalization and blowing up singular points, see \cite[Theorem 3.44, Chapter 9]{liu2}. Then, the dual graph of the special fiber is the Berkovich skeleton.
Instead of this approach, we study this problem using divisors on trees to directly construct the Berkovich skeleton.
The problem of computing the Berkovich skeleton for genus 2 curves was first studied in \cite{liu} in terms of semistable models. This was done systematically by studying the ramification data in \cite{RSS} and using Igusa invariants in \cite{igusa}. 
In the case of hyperelliptic curves, this problem was studied in \cite{trophyp} and later solved in \cite{bbc} using ramification data and admissible covers. In \cite{tropabelian}, Helminck presents criteria to reconstruct Berkovich skeleta using Laplacians on metric graphs. In this paper, we apply these techniques to the superelliptic case. 

The paper is structured as follows. In Section \ref{galoiscovers} we define Galois covers of trees and provide a definition for \emph{superelliptic coverings of metric graphs}. In Section \ref{galoiscovers2}, we consider the algebraic side and prove certain inertia groups are preserved under reduction. We furthermore prove that algebraic Galois covers yield Galois covers of metric graphs. 
In Section \ref{tropalgo} we present the algorithm (Theorem \ref{tropalgothm}) for computing the Berkovich skeleton of a superelliptic curve, taking as input the factored equation $y^n = \prod(x-\alpha_i)$. We apply this algorithm to compute the Berkovich skeleta of the genus 3 and genus 4 Shimura-Teichm\"{u}ller curves, and provide an example of a superelliptic curve which tropicalizes to the complete bipartite graph $K_{3,3}$. In Section \ref{realizability} we show the following realizability theorem for tropical superelliptic coverings.
\begin{theorem}[{\bf{Realizability}}]
Let $p$ be a prime number. A covering of metric graphs $\phi_\Sigma:\Sigma \rightarrow T$ is a superelliptic covering of degree $p$ of weighted metric graphs with rational edge lengths
if and only if there exists a superelliptic covering $\phi:X\rightarrow \mathbb{P}^1$ of degree $p$ tropicalizing to it. 
\end{theorem}
\noindent Lastly, in Section \ref{modulispaces} we study the locus of tropical superelliptic curves inside the moduli space $M_g^\text{tr}$, providing computations when possible of the number of maximal dimensional cones in this stacky polyhedral fan.

\section{Galois Covers of Trees}
\label{galoiscovers}
We first give the groundwork for discussing morphisms of metric graphs, and define superelliptic coverings of metric graphs. For the material on harmonic morphisms, we follow \cite{trophyp,admcov}. For a thorough introduction to divisors on graphs and metric graphs, we recommend \cite{baker}. 
Let $H$ be a connected graph and $l: E(H) \rightarrow \mathbb{R}_{>0}\cup\{\infty\}$ be a length function on the edges of $H$. Then a \emph{metric graph} is a connected metric space $\Sigma$ obtained by viewing edges $e$ in $H$ as line segments of length $l(e)$. We require that if $l(e) = \infty$, then one of the vertices of $e$ has degree one in $H$.
The pair $(H,l)$ is called a \emph{model} for $\Sigma$. A \emph{weighted metric graph} is a metric graph $\Sigma$ together with a weight function on its points $w:\Sigma \rightarrow \mathbb{Z}_{\geq 0}$ with finite support. We call edges of infinite length \emph{infinite leaves}, and these only meet the rest of the graph in one endpoint. If we denote by $b_1(\Sigma)$ the first Betti number of $\Sigma$, then we define the \emph{genus} of a weighted metric graph {($\Sigma$,~$w$)} to be 
$
 \sum_{v \in \Sigma} w(v) + b_1(\Sigma).
$
 A genus 0 weighted metric graph is a \emph{tree}. 
 
 \begin{figure}[h]
 \begin{center}
 \includegraphics[height = 2 in]{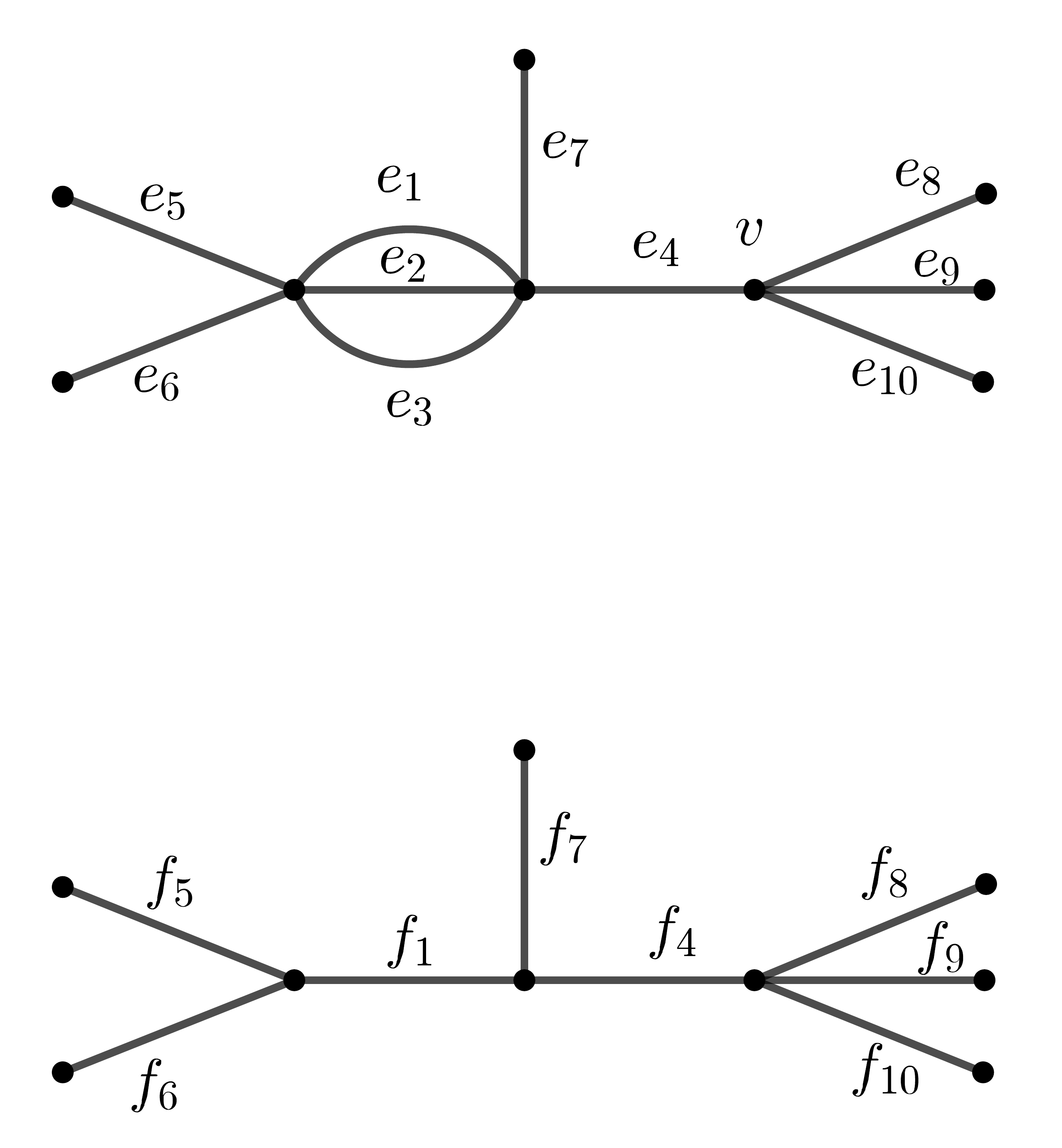}
 \caption{A superelliptic covering of metric graphs.}
 \label{FIG:harmonic_ex}
 \end{center}
 \end{figure}
 
 \begin{exa}
 \label{EX:harmonic_ex_1}
Consider top graph of Figure \ref{FIG:harmonic_ex}. This is a model for a metric graph with $l(e_i) = 1$ for $i \in \{1,2,3,4\}$ and otherwise $l(e_i) = \infty$. We assign weight 0 to every vertex except the vertex labeled $v$, and here we declare $w(v)=2$. The genus of this weighted metric graph is 4.
 \end{exa}
 
If $(H,l)$ and $(H',l')$ are loopless models for metric graphs $\Sigma$ and $\Sigma'$ then a \emph{nondegenerate morphism of loopless models} $\theta: (H,l) \rightarrow (H',l')$ consists of two maps $V(H) \rightarrow V(H')$ and $E(H) \rightarrow E(H')$ and a positive integer $s_e$ called the \emph{scaling factor} for each edge $e \in E(H)$ such that
\begin{enumerate}
\item  If $\theta(e) = e'$, then the vertices of $e$ must map to vertices of $e'$.
\item If $\theta(e) = e'$ where $l(e') \neq \infty$, then $l'(e')/l(e)=s_e$. 
\end{enumerate}
We will sometimes simply refer to this as a morphism of metric graphs $\Sigma \rightarrow \Sigma'$.
Let $v \in V(H)$ and $e' \in E(H')$. Define the \emph{local degree} to be
$$
\label{localdegree}
d_{v,e'} = \sum_{\substack{
e\ni v,\\ \phi(e) = e'
}} s_e.
$$
We say $\theta$ is \emph{harmonic} if for every $v \in V(H)$, and $e',e'' \in E(H')$ with $\phi(v) \in e',e''$ we have $d_{v,e'} = d_{v,e''}$. In this case we will always abbreviate the local degree by $d_v$, since it does not depend on the choice of edge $e'$. The \emph{degree} of a harmonic morphism is defined as 
$$
\sum_{\substack{e \in E(H),\\ \phi(e) = e'}} s_e.
$$
As in \cite{admcov}, we say that $\theta$ satisfies the \emph{local Riemann-Hurwitz condition} at $v$ if  
$$
\label{localrh}
2 - 2w(v) = d_v(2 - 2w'(\theta(v))) - \sum_{e\ni v}\left(s_e - 1\right).
$$

 \begin{exa}
 \label{EX:harmonic_ex_2}
Here, we study both graphs from Figure \ref{FIG:harmonic_ex}. We make a harmonic morphism $\phi$ from the graph on the top to the graph on the bottom by assigning $\phi(e_i) = f_1$, $s_{e_i} = 1$ for $i \in \{1,2,3\}$ and otherwise $\phi(e_i) = f_i$ and $s_{e_i} = 3$. All local degrees are 3 and the degree of $\phi$ is also 3. One can check that the local Riemann-Hurwitz conditions hold everywhere. For example, at $v$ we find
$$
2 - 2 \cdot 2 = 3 \cdot (2 - 2 \cdot 0) - 4 \cdot(3-1).
$$
\end{exa}

\begin{mydef}
\label{galoisdef} If $(H,l)$ is a model for $\Sigma$, then an \emph{automorphism} of $(H,l)$ is a harmonic morphism ${\theta: (H,l) \rightarrow (H,l)}$ of degree 1. Given a finite subgroup $G$ of $\text{Aut}((H,l))$, the \emph{quotient graph} $\Sigma / G$ has a model $(H / G, l_G)$ whose vertices are the $G$-orbits of $V(H)$ and whose edges are the $G$-orbits of edges defined by vertices lying in distinct $G$-orbits. For any edge $e\in{H}$, its $G$-orbit is denoted by $G{e}$. For edges $e \in H$, we set $s_e= |\text{Stab}(e)|$ and $l_G(G{e}) = s_e l(e)$. This is independent of the representative $e$, because the stabilizer of any other edge $g(e)$ in the orbit of $e$ is given by $g(\mathrm{Stab}(e))g^{-1}$, which has the same order as $\mathrm{Stab}(e)$.  Then, the quotient map $(H,l) \rightarrow (H/G, l_G)$ sending a vertex to its orbit and an edge to the edge between the orbits of its vertices is a nondegenerate harmonic morphism as long as the vertices of every edge lie in distinct $G$ orbits. For any finite subgroup $G$ of $\text{Aut}((H,l))$, we call a nondegenerate harmonic morphism $\Sigma\rightarrow{\Sigma/G}$ a \emph{Galois covering} of metric graphs if it satisfies the local Riemann-Hurwitz conditions at every $v$. The group $G$ is the \emph{Galois group} of the covering.
\end{mydef}

\begin{mydef}
A nondegenerate harmonic morphism $\theta: \Sigma \rightarrow T$ is a {\emph{superelliptic}} covering of metric graphs if $\theta$ is a Galois covering of metric graphs with Galois group $G:=\mathbb{Z}/n\mathbb{Z}$ and $T$ is a tree.\end{mydef}

 \begin{exa}
The morphism described in Example \ref{EX:harmonic_ex_2} and depicted in Figure \ref{FIG:harmonic_ex} is a superelliptic covering of metric graphs. The action of $G = \mathbb{Z}/3\mathbb{Z}$ on the graph permutes the edges $e_1,e_2,e_3$.
 \end{exa}

A \emph{divisor} on $\Sigma$ is a map $\Sigma \rightarrow \mathbb{Z}$ with finite support. 
A \emph{rational function} on a metric graph $\Sigma$ is a continuous, piecewise linear function $\psi : \Sigma \rightarrow \mathbb{R}$ which has integer slopes.
Given a rational function $\psi : \Sigma \rightarrow \mathbb{R}$ and a point $P \in \Sigma$, let $\sigma_P(\psi)$ be the sum of the slopes of $\psi$ in all outgoing directions at $P$. The \emph{principal divisor} of $\psi$ is the function
$\Delta(\psi): \Sigma \rightarrow \mathbb{Z}$, where $P \mapsto -\sigma_P(\psi)$.

 \begin{lemma} 
 \label{slopelemma}
Let $T$ be a metric tree (genus 0 metric graph), $v_0$ a point in $T$, and $\Delta(\psi)$ be a principal divisor on $T$. If $e$ is an interval in $T$ with endpoints $\{v_0,v_1\}$, and $e^0 = e \backslash \{v_0,v_1\}$ is deleted from $T$, let $T_e$ denote the  connected component of $T$ not containing $v_0$.
Then the magnitude of the slope of the rational function $\psi$ along $e$ is equal to $\sum_{x \in T_e} \Delta(\psi)(x)$.
 \end{lemma}
 \begin{proof}
Let $x$ and $y$ be the vertices of $e$ and suppose first $x\not = v_0$ is a leaf. Let $\psi_e$ be the slope of $\psi$ along the edge $e$. Then, $\Delta(\psi)(x) = -\sigma_x(\psi) =\psi_e$. Now, we proceed by induction on the number of edges in the tree $T$. Suppose $y$ is on the path from $x$ to $v_0$. We have
$
\Delta(\psi)(x) = -\sum_{e \ni x} \psi_{e'}.
$
Isolating $\psi_e$, we find
$$
 -\psi_e = \Delta(\psi)(x) + \sum_{e'\neq e, e'\ni x} \psi_{e'}. 
$$
Using the inductive assumption, we can solve for the $\psi_{e'}$ and arrive at the result.
 \end{proof}

\section{Galois coverings of semistable models and inertia groups}\label{galoiscovers2}
In this section, we show \emph{disjointly branched morphisms} of semistable models yield Galois coverings of metric graphs. Furthermore, we discuss inertia groups and prove they are preserved on reduction to the special fiber (Proposition \ref{ramind2}), 
allowing us to relate ramification degrees on a two dimensional scheme (that is, $\mathcal{X}$, a model for $X$) to those on a one dimensional scheme (a component in the special fiber $\mathcal{X}_s$). We use this equality in Section \ref{tropalgo} to reconstruct the Berkovich skeleton for superelliptic covers.

\subsection{Disjointly branched morphisms and inertia groups}\label{disbransect}

Let $\phi:X\rightarrow{Y}$ be a finite morphism of smooth, projective, geometrically connected curves over $K$. We say $\phi$ is \emph{Galois} if the corresponding morphism on function fields $K(Y)\rightarrow{K(X)}$ is Galois. That is, it is normal and separable.
 By a model $\mathcal{Y}$ for a curve $Y$, we mean an integral normal projective scheme $\mathcal{Y}$ of dimension $2$ with a flat morphism $\mathcal{Y}\rightarrow{\text{Spec}(R)}$, and an isomorphism $\mathcal{Y}_{\eta}\rightarrow{Y}$ of the generic fibers. Let $\mathcal{X}$ be a model for $X$ and $\mathcal{Y}$ be a model for $Y$. A \emph{finite morphism of models for} $\phi$ is a finite morphism $\mathcal{X}\rightarrow{\mathcal{Y}}$ such that the base change to $\text{Spec}(K)$ gives $\phi:X\rightarrow{Y}$.
\begin{mydef}\label{disbran}
Let $\phi:X\rightarrow{Y}$ be a finite, Galois morphism of curves over $K$ with Galois group $G$.
Let $\phi_{\mathcal{X}}:\mathcal{X}\rightarrow{\mathcal{Y}}$ be a finite morphism of models for $\phi$. We say $\phi_{\mathcal{X}}$ is \emph{disjointly branched} if the following hold:
\begin{enumerate}
\item The closure of the branch locus in $\mathcal{Y}$ consists of disjoint, smooth sections over $\text{Spec}(R)$.
\item The induced morphism $\mathcal{O}_{\mathcal{Y},\phi(y)}\rightarrow{\mathcal{O}_{\mathcal{X},y}}$ is \'{e}tale for every $y$ a generic point of an irreducible component in the special fiber of $\mathcal{X}$. 
\item The model $\mathcal{Y}$ is strongly semistable, meaning that $\mathcal{Y}$ is semistable and that the components in the special fiber are all smooth. 
\end{enumerate}
A theorem by Liu and Lorenzini \cite[Theorem 2.3]{liu_lorenzini_1999} says if $\phi_\mathcal{X}$ is disjointly branched then $\mathcal{X}$ is actually also \emph{semistable} and \cite[Proposition 4.1.1]{tropabelian} shows $\mathcal{X}$ is also strongly semistable.
\end{mydef}

The \emph{intersection graph} $\Sigma(\mathcal{X})$ of $\mathcal{X}$ has vertices corresponding to the irreducible components of the special fiber and edges corresponding to intersection points of two components in the special fiber. 
We weight each vertex with the genus of the corresponding component. We now study the action of $G$ on $\Sigma(\mathcal{X})$.

\begin{theorem}\cite[Theorem 4.6.1.]{tropabelian}\label{tropabeliantheorem}
Let $\phi_{\mathcal{X}}: \mathcal{X}\rightarrow \mathcal{Y}$ be a disjointly branched morphism of models for a finite Galois morphism $\phi: X \rightarrow Y$ of curves, with Galois group $G$.
There is a natural action of $G$ on the intersection graph $\Sigma(\mathcal{X})$ 
and the induced 
morphism of graphs
$
\Sigma(\mathcal{X})\rightarrow{\Sigma(\mathcal{Y})}
$
coincides with the quotient map $\Sigma(\mathcal{X})\rightarrow{\Sigma(\mathcal{X})/G}$.
\end{theorem}
For now, this is a statement about graphs; we give the result for metric graphs in Section \ref{algmetgraph}.

\begin{mydef}
\label{DEF:id}
Let $G$ be a finite group acting on a scheme $Z$. 
For any point $z$ of $Z$, we define the decomposition group $D_{z,Z}$ to be $\{\sigma \in G:\sigma(z)=z\}$, the stabilizer of $z$. 
Every element $\sigma\in{D_{z,Z}}$ naturally acts on the structure sheaf $\mathcal{O}_{Z,z}$ and the residue field $k(z)$. We define the \emph{inertia group} $I_{z,Z}$ of $z$ to be the elements of $D_{z,Z}$ reducing to the identity on $k(z)$. In other words, $\sigma \in I_{z,Z}$ if and only if for every $f \in {\mathcal{O}_{Z,z}}$, we have $\sigma f \equiv f \mod m_z$, where $m_{z}$ is the unique maximal ideal of $\mathcal{O}_{Z,z}$.
When context is clear, we omit the $Z$ in $I_{z,Z}$ and $D_{z,Z}$.
\end{mydef}

Suppose we have a normal integral scheme $W$ with function field $K(W)$ and a finite Galois extension $L$ of $K(W)$ with Galois group $G=\text{Gal}(L/K(W))$. We take the normalization $W'$ of $Y$ in $L$ (which we now write as $K(W')$) to obtain a morphism of normal integral schemes $W'\rightarrow{W}$. In fact, we have $W=W'/G$ 
(See \cite[Proposition 4.2.4.]{tropabelian}). We now study this quotient locally. %

\begin{mydef}\label{FiniteGaloisExtension}
Let $A$ be a normal domain with fraction field $K$, let $L$ be a finite separable extension of $K$ and let $B$ be the integral closure of $A$ in $L$. We say that $B$ is an {\it{ integrally closed extension}} of $A$. If additionally $L$ is a Galois extension of $K$, we say that $B$ is an {\it{integrally closed Galois extension}} of $A$.
\end{mydef}
\begin{lemma}\label{LemmaNormal}
Let $B\supseteq{A}$ be an integrally closed extension as in Definition \ref{FiniteGaloisExtension}. Then $B$ is normal.
\end{lemma}
\begin{proof}
See \cite[Corollary 5.5]{Atiyah}.
\end{proof}
Let $B\supseteq{A}$ be an integrally closed Galois extension. For every $x$ in $B$ satisfying an integral relation over $A$, we find that $\sigma(x)$ satisfies the same integral relation. This implies that there is a natural homomorphism $G\rightarrow{\mathrm{Aut}_{A}(B)}$. For any subgroup $H\subset{G}$, we then denote the ring of invariants by
$$
B^{H}=\{x\in{B}:\sigma(x)=x\text{ for every }\sigma\in{H}\}.
$$

\begin{lemma}\label{NormalityRings}
Let $B\supseteq{A}$ be an integrally closed Galois extension with Galois group $G$. Then,
\begin{enumerate}
\item For any subgroup $H\subseteq{G}$, the ring $B^{H}$ is normal,
\item The ring $B^{H}$ is the integral closure of $A$ in the field of fractions of $B^{H}$, and
\item $B^{G}=A$.
\end{enumerate}
\end{lemma}
\begin{proof}
Let $L^{H}$ be the field of fractions of $B^{H}$. An easy verification then shows that $L^{H}\cap{B}=B^{H}$. Let $x\in{L^{H}}$ be integral over $B^{H}$. Viewing $x$ as an element of $L$, we find that $x\in{B}$ by Lemma \ref{LemmaNormal}. But then $x\in{L^{H}\cap{B}}=B^{H}$, as desired.

For the second part, let $C$ be the integral closure of $A$ in $L^{H}$ and let $x\in{C}$. Then $x\in{B}\cap{L^{H}}=B^{H}$, so $C\subseteq{B^{H}}$. Let $x\in{B^{H}}$. Then $x\in{B}$, so $x$ is integral over $A$. Thus $x\in{C}$.

For the last part, first note that $B^{G}$ contains $A$. Furthermore, let $x\in{B^{G}}$. Since $x\in{B}$ is integral over $A$ and $x\in{K}=L^{G}$, we find that $x\in{A}$ by normality of $A$. This finishes the proof.
\end{proof}

\begin{rem}\label{Noetherian}
If we suppose that $A$ is Noetherian, then any integrally closed extension $B\supseteq{A}$ as in Definition \ref{FiniteGaloisExtension} is finite as an $A$-module by \cite[Chapter 4, Proposition 1.25]{liu2}.
\end{rem}

Let $A\subset{B}$ be an integrally closed Galois extension. By the anti-equivalence of commutative rings and affine schemes, we obtain a morphism $\mathrm{Spec}(B)\rightarrow{\mathrm{Spec}(A)}$. Moreover, the action of the Galois group $G$ on $B$ naturally extends to a group action of $G$ on $\mathrm{Spec}(B)$ such that $\mathrm{Spec}(B)/G=\mathrm{Spec}(A)$, see \cite[Proposition 1.1, Page 88]{SGA1}. For any prime $\mathfrak{q}\in\mathrm{Spec}(B)$, we can thus consider its decomposition and inertia groups, as introduced in Definition \ref{DEF:id}. We will write $I_{\mathfrak{q}}$ and $D_{\mathfrak{q}}$ for the inertia and decomposition groups of $\mathfrak{q}$ and $k(\mathfrak{q})$ for its residue field. We now have the following result.
\begin{lemma}\label{exactseq}
Let $B\supseteq{A}$ be an integrally closed Galois extension with Galois group $G$. If $\mathfrak{q}\in\text{Spec}(B)$ is a prime lying over $\mathfrak{p}\subset{A}$, then $k(\mathfrak{q})/k(\mathfrak{p})$ is an algebraic normal extension and the following is exact:
$$
1\rightarrow{I_{\mathfrak{q}}}\rightarrow{}D_{\mathfrak{q}}\rightarrow{\text{Aut}(k(\mathfrak{q})/k(\mathfrak{p}))}\rightarrow{1},
$$
where $I_\mathfrak{q}$ is the inertia group of $\mathfrak{q}$ and $D_\mathfrak{q}$ is the decomposition group of $\mathfrak{q}$ (see Definition \ref{DEF:id}). 
\end{lemma}
\begin{proof}
This is \cite[Tag 0BRK]{stacks-project}, where $I_q$ is the kernel of the surjective morphism described there.
\end{proof}
In our case, as in Definition \ref{disbran}, the extension of residue fields is always Galois; 
our assumption that the degree of the Galois extension is relatively prime to the characteristic of the residue field implies separability.

We now show inertia groups directly measure and control ramification. This is the content of Proposition~\ref{ramprop} which we will use to relate inertia groups on $\mathcal{X}$ to inertia groups on the special fiber. 
We start by studying the inertia group $I_{\mathfrak{q}}$ and the invariant ring $B^{I_{\mathfrak{q}}}$ a bit closer. By Lemma \ref{NormalityRings}, we have that $B^{I_{\mathfrak{q}}}$ is normal. If we furthermore assume that $A$ is Noetherian, then $B^{I_{\mathfrak{q}}}$ is also finite as an $A$-module, see Remark \ref{Noetherian}.
\begin{lemma}\label{oneprime}
Let $B\supseteq{A}$ be an integrally closed Galois extension and let $\mathfrak{q}\in\mathrm{Spec}(B)$ be any prime ideal. Let $j:B^{I_{\mathfrak{q}}}\rightarrow{B}$ be the natural inclusion map and let $j^{*}:\text{Spec}(B)\rightarrow{\text{Spec}(B^{I_{\mathfrak{q}}})}$ be the corresponding map on the spectra. Then $(j^{*})^{-1}(j^{*}(\mathfrak{q}))=\{\mathfrak{q}\}$.
\end{lemma}
\begin{proof}
The morphism $j^{*}$ coincides with the quotient map $\text{Spec}(B)\rightarrow{\text{Spec}(B)/I_{\mathfrak{q}}}$ (See \cite[Expos\'{e} V, Proposition 1.1., Page 88]{SGA1}). This means any other prime mapping to $\mathfrak{q}\cap{B^{I_{\mathfrak{q}}}}$ is of the form $\sigma(\mathfrak{q})$ for some $\sigma$ in $I_{\mathfrak{q}}$. But for those $\sigma$, we have $\sigma(\mathfrak{q})=\mathfrak{q}$. This concludes the proof. 
\end{proof}

\begin{lemma}\label{resfieldinert}
Let $B\supseteq{A}$ be an integrally closed Galois extension with Galois group $G$ and let $\mathfrak{q}\in\mathrm{Spec}(B)$ be any prime ideal. Let $B\supset{B^{I_{\mathfrak{q}}}}$ be the natural inclusion, and let $k(\mathfrak{q})^{\text{sep}}$ be the separable closure of $k(\mathfrak{q}\cap{A})$ in $k(\mathfrak{q})$. 
Then $(k(\mathfrak{q}))^{\text{sep}}=k(\mathfrak{q}\cap{B^{I_{\mathfrak{q}}}})$. 
\end{lemma}
\begin{proof}
Consider the Galois extension $L\supset{L^{I_{\mathfrak{q}}}}$ with Galois group $I_{\mathfrak{q}}$. We find $I_{\mathfrak{q}}=D_{\mathfrak{q}}$ and so the automorphism group ${\text{Aut}(k(\mathfrak{q})/k(\mathfrak{q}\cap{B^{I_{\mathfrak{q}}}}))}$ is trivial by Lemma \ref{exactseq}. This automorphism group is isomorphic to the Galois group of the separable closure of $k(\mathfrak{q}\cap{B^{I_{\mathfrak{q}}}})$ in $k(\mathfrak{q})$. By Galois theory, $k(\mathfrak{q}\cap{B^{I_{\mathfrak{q}}}})$ is separably closed inside $k(\mathfrak{q})$. By \cite[Proposition 2.2., page 92]{SGA1}, the extension $B^{I_{\mathfrak{q}}}\supset{A}$ is \'{e}tale at $\mathfrak{q}\cap{B^{I_{\mathfrak{q}}}}$, so the residue field extension $k(\mathfrak{q}\cap{}B^{I_{\mathfrak{q}}})\supset{k(\mathfrak{q}\cap{A})}$ is separable by \cite[Tag 00U3]{stacks-project} and the fact that \'{e}tale morphisms are stable under base change. Thus $k(\mathfrak{q}\cap{B^{I_{\mathfrak{q}}}})\subseteq{(k(\mathfrak{q}))^{\text{sep}}}$, and every element of $k(\mathfrak{q})$ that is separable over $k(\mathfrak{q}\cap{A})$ is also separable over the field $k(\mathfrak{q}\cap{}B^{I_{\mathfrak{q}}})$. We thus find $(k(\mathfrak{q}))^{\text{sep}}=k(\mathfrak{q}\cap{B^{I_{\mathfrak{q}}}})$, as desired. 
\end{proof}
\begin{cor}\label{corchar}
Let $B\supseteq{A}$ be an integrally closed Galois extension and let $\mathfrak{q}\in\mathrm{Spec}(B)$ be any prime ideal. Suppose $\text{char}(k(\mathfrak{q}))\nmid{|G|}$. Then $k(\mathfrak{q})=k(\mathfrak{q}\cap{B^{I_{\mathfrak{q}}}})$.
\end{cor}
Inertia groups measure ramification, as the following proposition describes.
\begin{prop}\label{ramprop}
Let $B\supseteq{A}$ be an integrally closed Galois extension with Galois group $G$ as in Definition \ref{FiniteGaloisExtension} and let $\mathfrak{q}\in\mathrm{Spec}(B)$ be any prime ideal. \begin{enumerate}
\item Consider the subring $B\supset{B^{I_{\mathfrak{q}}}}\supset{A}=B^{G}$. Then $B^{I_{\mathfrak{q}}}\supset{A}$ is \'{e}tale at $\mathfrak{q}\cap{}B^{I_{\mathfrak{q}}}$.
\item More generally, consider any subgroup $H$ of $G$. Then $B^{H}\supset{A}$ is \'{e}tale at $\mathfrak{q}\cap{B^{H}}$ if and only if $H\supseteq{I_{q}}$. 
\end{enumerate}
\end{prop}
\begin{proof}
The fact that if $H\supseteq{I_{\mathfrak{q}}}$, then $B^{H}\supset{A}$ is \'{e}tale at $\mathfrak{q}\cap{B^{H}}$ is \cite[Proposition 2.2., page 92]{SGA1}. Taking $H=I_{\mathfrak{q}}$ also proves the first statement. To prove the inclusion $H\supseteq{I_{\mathfrak{q}}}$, we use the material in \cite[\href{https://stacks.math.columbia.edu/tag/0BSK}{Section 0BSK}]{stacks-project} on strict Henselizations. We first note that we can localize $A$ and assume that it is local with 
maximal ideal $\mathfrak{p}=\mathfrak{q}\cap{A}$. The strict Henselization of the triple $(A,\mathfrak{p},k(\mathfrak{p}))$ is then given by $A^{\mathrm{sh}}:=\lim_{(S,\mathfrak{p}_{S},\alpha)}S$, where the limit is taken over the filtered category of all triples $(S,\mathfrak{p}_{S},\alpha)$, where $A\rightarrow{S}$ is an \'{e}tale ring map, $\mathfrak{p}_{S}$ is a prime ideal in $S$ lying above $\mathfrak{p}$ and $\alpha:k(\mathfrak{p}_{S})\rightarrow{k(\mathfrak{p})^{\mathrm{sep}}}$ is an embedding of residue fields, see \cite[\href{https://stacks.math.columbia.edu/tag/04GN}{Lemma 04GN}]{stacks-project}. Now consider the integral closure $A^{\mathrm{sep}}$ of $A$ in the separable closure $K^{\mathrm{sep}}$. We choose a prime ideal $\mathfrak{q}^{\mathrm{sep}}$ in $A^{\mathrm{sep}}$ lying over $\mathfrak{q}\in{B}$. The localization $A^{\mathrm{sep}}_{\mathfrak{q}^{\mathrm{sep}}}$ is then a strictly Henselian ring, and this induces an inclusion $A^{\mathrm{sh}}\rightarrow{A^{\mathrm{sep}}_{\mathfrak{q}^{\mathrm{sep}}}}$, see the beginning of \cite[\href{https://stacks.math.columbia.edu/tag/0BSD}{Section 0BSD}]{stacks-project}. As in the proof of \cite[\href{https://stacks.math.columbia.edu/tag/0BSW}{Lemma 0BSW}]{stacks-project}, we then have 
\begin{equation}
    A^{\mathrm{sh}}=(A^{\mathrm{sep}}_{\mathfrak{q}^{\mathrm{sep}}})^{I_{\mathfrak{q}^{\mathrm{sep}}}},
\end{equation}
where $I_{\mathfrak{q}^{\mathrm{sep}}}$ is the inertia group as defined in Definition \ref{DEF:id}. 

We now relate this material to our subgroup $H\subset{G}$. In terms of infinite Galois theory, $H$ corresponds to a closed (and open) subgroup $G_{L}\subset{}\overline{H}\subset{G}_{K}$ such that $\overline{H}/G_{L}=H$. Similarly, we have a natural surjective map
\begin{equation}
    I_{\mathfrak{q}^{\mathrm{sep}}}\rightarrow{I_{\mathfrak{q}}}
\end{equation}
obtained by restricting automorphisms to $L$, see \cite[\href{https://stacks.math.columbia.edu/tag/0BSX}{Lemma 0BSX}]{stacks-project}. The kernel of this map is  $G_{L}\cap{I_{\mathfrak{q}^{\mathrm{sep}}}}$, inducing the isomorphism $I_{\mathfrak{q}}=I_{\mathfrak{q}^{\mathrm{sep}}}/(G_{L}\cap{I_{\mathfrak{q}^{\mathrm{sep}}}})$. 
By our assumption on $H$, we have that $B^{H}\supset{A}$ is \'{e}tale at $\mathfrak{q}_{H}:=\mathfrak{q}\cap{B^{H}}$. This means that there exists a $g\in{B^{H}}\backslash{\mathfrak{q}_{H}}$ such that the localization $A\to{(B^{H})_{g}}$ is \'{e}tale. After choosing an embedding $k(\mathfrak{q}_{H})\to{k(\mathfrak{p}^{\mathrm{sep}})}$, we then obtain maps \begin{equation}
    A\rightarrow{}(B^{H})_{g}\rightarrow{A^{\mathrm{sh}}}=(A^{\mathrm{sep}}_{\mathfrak{q}^{\mathrm{sep}}})^{I_{\mathfrak{q}^{\mathrm{sep}}}}.
\end{equation} 

These are easily seen to be injective. Since the fraction field of $(B^{H})_{g}$ is the invariant field of $\overline{H}$, we have 
$\overline{H}\supset{I_{\mathfrak{q}^{\mathrm{sep}}}}$ by infinite Galois theory. We then obtain an inclusion 
\begin{equation}
    I_{\mathfrak{q}}=I_{\mathfrak{q}^{\mathrm{sep}}}/(G_{L}\cap{I_{\mathfrak{q}^{\mathrm{sep}}}})\subseteq{\overline{H}/G_{L}=H},
\end{equation}
as desired. This finishes the proof. 
\end{proof}

For any subfield $K\subseteq{K'}\subseteq{L}$, we can write $K'=L^{H}$ for some subgroup $H$ of $G$ by Galois theory. By Proposition \ref{ramprop}, $B^{H}\supseteq{B^{G}=A}$ is ramified at some point $x'$ if and only if the inertia group $I_{x}$ of some point $x$ lying above is {\it{not}} contained in $H$. In other words, we can describe ramification in terms of Galois theory.
This criterion turns out to be very useful in relating different inertia groups $I_{x}$ and $I_{y}$ for points $x$ and $y$ in $\text{Spec}(B)$. For instance, if $B^{I_{x}}\supset{A}$ is \'{e}tale at the image of $y$ in $B^{I_{x}}$, and $B^{I_{y}}$ is \'{e}tale at the image of $x$ in $B^{I_{y}}$, then $I_{y}=I_{x}$. We use this in Section \ref{compinert}.  
We make two further assumptions on the morphism $\phi_{\mathcal{X}}:\mathcal{X}\rightarrow{\mathcal{Y}}$ of models over $\text{Spec}(R)$. 
We assume the ramification points of $\phi:X\rightarrow{Y}$ are rational over $K$. 
 We assume the residue field $k$ is large enough so that for every intersection point $x\in\mathcal{X}$, we have $D_{x}=I_{x}$.

\begin{exa}
Let us find out what these decomposition and inertia groups are for a disjointly branched morphism $\phi_{\mathcal{X}}:\mathcal{X}\rightarrow{\mathcal{Y}}$ of semistable models. 
\begin{enumerate}
\item Let $x\in\mathcal{X}$ be an intersection point. That is, it is a point such that $\hat{\mathcal{O}}_{\mathcal{X},x}\simeq{R[[u,v]]/(uv-\pi^{n})}$ for some positive integer $n$. By our second assumption in Definition \ref{disbran} we have $D_{x}=I_{x}$. Since the action of $G$ is transitive on the edges lying above $\phi_{\mathcal{X}}(x)$, there are $|G|/|I_{x}|$ edges lying above $\phi_{\mathcal{X}}(x)$.  
\item Let $x$ be the generic point of an irreducible component $\Gamma_{x}$ in the special fiber $\mathcal{X}_{s}$. Let $y$ and $\Gamma_{y}$ be their respective images in $\mathcal{Y}_{s}$. By our second assumption for disjointly branched morphisms, the inertia group $I_{x}$ is trivial. Thus, the decomposition group can be identified with the automorphisms of the function field $k(x)$ of the component $\Gamma_{x}$ fixing the function field $k(y)$ of the component $\Gamma_{y}$, by Lemma~\ref{exactseq}. This implies $\Gamma_{x}/D_{x}=\Gamma_{y}$ as curves over the residue field, since morphisms of smooth curves are determined by the corresponding inclusions of function fields. So, $\Gamma_{x}$ and $\Gamma_{y}$ are smooth, since $\mathcal{X}$ and $\mathcal{Y}$ are strongly semistable.

\item Let $x\in{X}$ be a generic ramification point. Then $|I_{x}|$ is the ramification degree. This follows from the fact that $x$ is totally ramified in the extension $L\supset{L^{I_{x}}}$, which has degree $|I_{x}|$.

Let us study an example where the decomposition group $D_{x}$ for a generic branch point is bigger than $I_{x}$.  
Take the Galois covering $(x,y)\mapsto{x}$ for the curve $X$ defined by $y^4=x^2\cdot{(x+2)}$ over $\mathbb{Q}(i)$. This is Galois with Galois group $G=\mathbb{Z}/4\mathbb{Z}$, where the action on fields comes from multiplication by $i$ on $y$ and the identity on $x$. Let us find the normalization of the algebra $A:=\mathbb{Q}(i)[x,y]/(y^4-x^2\cdot{(x+2)})$. 
The integral element $z=y^2/x$ satisfies 
$z^2=x+2$.
The maximal ideal $\mathfrak{m}=(x,y,z^2-2)$ is then locally principal with generator $y$, as we can write
\begin{equation*}
z^2-2=x=\dfrac{y^2}{z}
\end{equation*}
in the localization of $A':=A[z]/(z^2-x-2)$ 
at $\mathfrak{m}$. Since $A$ was already normal at the other primes (by the Jacobi criterion for instance), $A'$ is the normalization. Here we use that a domain is normal if and only if it is normal at all its localizations.

The Galois group $\mathbb{Z}/4\mathbb{Z}$ fixes $\mathfrak{m}$, so $D_{\mathfrak{m}}=G$. By inspecting the action on the residue field, we have $|I_{\mathfrak{m}}|=2$. Indeed, the automorphism defined by $\sigma(y)=iy$ sends $z$ to $-z$, which is nontrivial on the residue field. 
In this case the decomposition group is strictly larger than the inertia group.

If we consider the curve over the field $\mathbb{Q}(i)(\sqrt{2})$, the situation changes. The above equations still define the normalization, but $\mathfrak{m}=(x,y,z^2-2)$ is no longer maximal. There are now two maximal ideals lying above $\mathfrak{m}_{0}=(x)$, namely $\mathfrak{m}_{\pm}=(x,y,z\pm\sqrt{2})$. We have $|I_{\mathfrak{m}_{\pm}}|=|D_{\mathfrak{m}_{\pm}}|=2$. For disjointly branched morphisms, we assume extensions of this form have already been made.
\end{enumerate}
\end{exa}

\subsection{Specializing divisors}\label{specialdivisor}
In this section, we define the specialization map $\rho:\mathrm{Div}(X)\rightarrow{\mathrm{Div}(\Sigma(\mathcal{X}))}$ for divisors on a curve $X$ and a strongly semistable regular model $\mathcal{X}$ for $X$. This then gives a divisor on the weighted metric graph associated to $\mathcal{X}$. This specialization map will be used to explicitly calculate the Laplacian of an element $f\in{K(X)}$. We will first define divisors and Laplacians on $\Sigma(\mathcal{X})$ in this discrete case. These definitions will be similar to the definitions given in Section \ref{galoiscovers} and we invite the reader to compare the discrete and continuous variants. 

Let $\mathcal{X}$ be a strongly semistable regular model for $X$ with intersection graph $\Sigma(\mathcal{X})$. Let  $\mathrm{Div}(\mathcal{X})$ be the group of Cartier divisors on $\mathcal{X}$ and let $\mathrm{Div}_{s}(\mathcal{X})$ be the subgroup of Cartier divisors with support on the special fiber $\mathcal{X}_{s}$. Let $\{\Gamma_{1},...,\Gamma_{r}\}$ be the set of irreducible components in $\mathcal{X}_{s}$. For any vertex $v\in\Sigma(\mathcal{X})$, we let $\Gamma_{v}$ be the associated irreducible component.

We define $\text{Div}(\Sigma(\mathcal{X}))$ to be the free abelian group on the vertices $V(\Sigma(\mathcal{X}))$ of $\Sigma(\mathcal{X})$. 
Writing $D\in{\text{Div}(\Sigma(\mathcal{X}))}$ as $D=\sum_{v\in{V(\Sigma(\mathcal{X}))}}c_{v}(v)$, we define the degree map as $\text{deg}(D)=\sum_{v\in{V(\Sigma(\mathcal{X}))}}c_{v}$. We let $\text{Div}^{0}(\Sigma(\mathcal{X}))$ be the group of divisors of degree zero on $\Sigma(\mathcal{X})$.

Let $\mathcal{M}(\Sigma(\mathcal{X}))$ be the group of $\mathbb{Z}$-valued functions on $V(\Sigma(\mathcal{X}))$. Define the {\emph{Laplacian operator}} $\Delta:\mathcal{M}(\Sigma(\mathcal{X}))\rightarrow\text{Div}^{0}(\Sigma(\mathcal{X}))$ by 
\begin{equation*}
\Delta(\phi)=\sum_{v\in{V(\Sigma(\mathcal{X}))}}\sum_{e=vw\in{E(\Sigma(\mathcal{X}))}}(\phi(v)-\phi(w))(v).
\end{equation*}   
The fact that the image of $\Delta$ is in $\text{Div}^{0}(\Sigma(\mathcal{X}))$ is contained in \cite[Corollary 1]{bakerfaber}, where it is found as a consequence of the self-adjointness of the Laplacian operator. It also follows in this case by a direct computation. 
We then define the group of principal divisors to be the image of the Laplacian operator:
\begin{equation*}
\text{Prin}(\Sigma(\mathcal{X})):=\Delta(\mathcal{M}(\Sigma(\mathcal{X}))).
\end{equation*}

We will now transport algebraic divisors in $\mathrm{Div}(X)$ to divisors on the graph $\Sigma(\mathcal{X})$. To do this, we will use the intersection pairing on $\mathcal{X}$. This is a bilinear pairing  
$$\mathrm{Div}(\mathcal{X})\times{\mathrm{Div}_{s}(\mathcal{X})}\rightarrow{\mathbb{Z}},$$ which we will denote (see \cite[Chapter 9, Theorem 1.12]{liu2})
by $(\mathcal{D}_{1}\cdot{\mathcal{D}_{2}})\in\mathbb{Z}$. This
then gives rise to the \emph{specialization} homomorphism $\rho:\mathrm{Div}(\mathcal{X})\rightarrow{\mathrm{Div}(\Sigma(\mathcal{X}))}$ defined by
\begin{equation}\label{SpecializationHomomorphism}
\rho(\mathcal{D})=\sum_{v\in\Sigma(\mathcal{X})}(\mathcal{D}\cdot{\Gamma_{v}})(v),
\end{equation}
see \cite[Section 2]{baker} and \cite[Section 2.2.1]{tropabelian}.

For any $D\in\mathrm{Div}(X)$, the closure inside $\mathcal{X}$ defines a Cartier divisor which we will denote by $\overline{D}$. This gives a homomorphism $\mathrm{Div}(X)\rightarrow{\mathrm{Div}(\mathcal{X})}$. Composing it with the specialization homomorphism, we obtain a homomorphism $\mathrm{Div}(X)\rightarrow{\mathrm{Div}(\Sigma(\mathcal{X}))}$ which we will again denote by $\rho$. 
We then have
\begin{lemma}
Let $\mathcal{X}$ be a strongly semistable model over $R$ for a smooth, proper, geometrically connected curve $X$, let $\Sigma(\mathcal{X})$ be its intersection graph and let $\rho:\mathrm{Div}(\mathcal{X})\rightarrow{\mathrm{Div}(\Sigma(\mathcal{X}))}$ be the specialization map from Equation \ref{SpecializationHomomorphism}. Then 
$\rho(\mathrm{Prin}(X))\subset{\mathrm{Prin}(\Sigma(\mathcal{X}))}$.
\end{lemma}
\begin{proof}
For any $f\in{K(X)^{*}}$, let $\mathrm{div}_{\eta}(f)$ be the induced divisor on $X$ and let $\mathrm{div}(f)$ be the induced divisor on $\mathcal{X}$. Write $D\in\mathrm{Prin}(X)\backslash{\{0\}}$ as $D=\mathrm{div}_{\eta}(f)$ for some $f\in{K(X)^{*}}$. We then have 
\begin{equation*}
\mathrm{div}(f)=\mathrm{div}_{\eta}(f)+V_{f},
\end{equation*}
where $V_{f}$ is the vertical divisor associated to $f$. By \cite[Chapter 9, Theorem 1.12.c]{liu2} we have that $\rho(\mathrm{div}(f))=0$. By \cite[Lemma 3.3.1]{tropabelian} we then have that $\rho(\Gamma_{v})\in\mathrm{Prin}(\Sigma(\mathcal{X}))$ for every $v\in\Sigma(\mathcal{X})$, so $\rho(V_{f})\in\mathrm{Prin}(\Sigma(\mathcal{X}))$. This then gives $\rho(\mathrm{div}_{\eta}(f))\in\mathrm{Prin}(\Sigma(\mathcal{X}))$, as desired.  
\end{proof}
In other words, for every $f\in{K(X)^{*}}$, there is a $\phi$ such that $\rho(\mathrm{div}_{\eta}(f))=\Delta(\phi)$. We now consider this specialization homomorphism for $K$-rational points $P\in{X(K)}$. By \cite[Chapter 9, Corollary 1.32]{liu2}, we find that $\rho(P)=(v_{\Gamma})$ for a unique irreducible component $\Gamma\subset{\mathcal{X}_{s}}$. Note that such a component need not exist for nonregular models.
\begin{exa}
Let $\mathcal{X}:=\mathrm{Proj}(R[X,T,W]/(XT-\pi^{2}W^2))$ where $R[X,T,W]/(XT-\pi^{2}W^2)$ has the usual grading. Let $U:=\mathrm{Spec}(R[x,t]/(xt-\pi^{2}))$ be the affine open of $\mathcal{X}$ given by $D_{+}(W)$. Note that $\mathcal{X}$ is not regular at the point $(x,t,\pi)$. Let $P=(x-\pi,t-\pi)$ be a $K$-rational point of $X$. The closure of $P$ is then given by $\{P,\overline{P}\}$, where $\overline{P}=(x-\pi,t-\pi,\pi)$. Note that $\overline{P}$ lies on both irreducible components of the special fiber. So, we do not directly have an associated divisor on the dual graph of~$\mathcal{X}$.  

The intersection graph $\Sigma(\mathcal{X})$ consists of two vertices $\{v_{1},v_{2}\}$ with an edge $e$ of length two. Taking a desingularization above this edge, we obtain a regular model $\mathcal{X}'$ with intersection graph consisting of three vertices $\{v_{1},v',v_{2}\}$ and two edges $\{e_{1},e_{2}\}$. Here $v_{1}$ and $v'$ are connected by $e_{1}$, and $v'$ and $v_{2}$ are connected by $e_{2}$. The original edge $e$ has been subdivided into  two edges $\{e_{1},e_{2}\}$ and the vertex $v'$ in this new intersection graph $\Sigma(\mathcal{X}')$. The point $P$ now specializes to the vertex $v'$ in the middle. For explicit equations defining this model $\mathcal{X}'$, see \cite[Chapter 8, Example 3.53]{liu2}. 
\end{exa}

\subsection{Comparing inertia groups}\label{compinert}
Consider a disjointly branched morphism $\phi_{\mathcal{X}}:\mathcal{X}\rightarrow{\mathcal{Y}}$. 
Let $\Gamma\subset{\mathcal{X}_{s}}$ be any irreducible component in the special fiber. There is a natural Galois morphism of smooth curves $\phi_{\Gamma}:\Gamma\rightarrow{\Gamma'}$ where $\Gamma'$ is an irreducible component in the special fiber $\mathcal{Y}_{s}$. This uses the condition on the characteristic of $k$. To see this, one can use \cite[Section 4.5]{tropabelian} or Lemma \ref{exactseq} and Corollary \ref{corchar}.
The morphism $\phi_{\Gamma}$ is induced by the natural Galois morphism of function fields $k(\Gamma')\rightarrow{k(\Gamma)}$, which is obtained from the natural map of discrete valuation rings $\mathcal{O}_{\mathcal{Y},y_{\Gamma'}}\rightarrow{\mathcal{O}_{\mathcal{X},y_{\Gamma}}}$. Here $y_{\Gamma}$ and $y_{\Gamma'}$ are the generic points of the irreducible components $\Gamma$ and $\Gamma'$.

We recall the definition of the natural \emph{reduction maps} associated to $\mathcal{X}$ and $\mathcal{Y}$.
Let us describe this map $r_{\mathcal{X}}$ for $\mathcal{X}$. Let $X^{0}$ be the set of closed points of the generic fiber $X$ of $\mathcal{X}$. For every closed point $x$, we can consider its closure $\overline{\{x\}}$ inside $\mathcal{X}$. This closure is then an irreducible scheme, finite over $\text{Spec}(R)$. Since $R$ is Henselian, $\overline{\{x\}}$ is local, giving a unique closed point. We let $r_{\mathcal{X}}(x)$ be this closed point. 
We now also extend this map to closed points of $\mathcal{X}$ for convenience.
Let $x\in\mathcal{X}$ be a closed point, considered as an $R$-scheme. We can take the base change $\{x\}\times_{\mathrm{Spec}(R)}{\mathrm{Spec}(k)}$. There is a closed immersion of this scheme into the special fiber $\mathcal{X}_{s}$ and we define its image in $\mathcal{X}_{s}$ to be $r_{\mathcal{X}}(x)$. In other words, we consider this closed point as a point on the special fiber. Note that intersection points also naturally define closed points in $\mathcal{X}$. Here an intersection point is a point $x\in\mathcal{X}$ such that $\hat{\mathcal{O}}_{\mathcal{X},x}\simeq{R[[u,v]]/(uv-\pi^{k})}$ for some $k\in\mathbb{Z}_{\geq{0}}$.

\begin{prop}\label{ramind2}
Let $\phi_{\mathcal{X}}:\mathcal{X}\rightarrow{\mathcal{Y}}$ be a disjointly branched morphism as in Definition \ref{disbran} and let $x\in\mathcal{X}$ be a closed point of the generic fiber or an intersection point. Let $\Gamma$ be any irreducible component in the special fiber $\mathcal{X}_{s}$ containing $r_{\mathcal{X}}(x)$ (as described in the preceding paragraph).
Then
$$
I_{x,\mathcal{X}}=I_{r_{\mathcal{X}}(x),\Gamma}$$
where the second inertia group is an inertia group of the Galois covering $\Gamma\rightarrow{\Gamma'}$ on the special fiber.
\end{prop}

\begin{proof}
We first let $z\in\mathcal{X}$ be any closed point and $y$ the generic point of $\Gamma$. 
We then have a natural injection $D_{z,\mathcal{X}}\rightarrow{D_{y,\mathcal{X}}}$. For $z$ smooth this follows directly from the fact that $y$ is the unique generic point under $z$. For $z$ an intersection point, this follows from 
\cite[Proposition 5.2.1]{tropabelian}. By Lemma \ref{exactseq}, $D_{y,\mathcal{X}}$ can be identified with the Galois group of the function field extension $k(\Gamma)\supset{k(\Gamma')}$. The image of $D_{x,\mathcal{X}}$ in this Galois group is then in fact equal to $D_{r_{\mathcal{X}}(x),\Gamma}$. We thus see $D_{x,\mathcal{X}}=D_{r_{\mathcal{X}}(x),\Gamma}$ for any closed point $x$. By our assumption on the residue field, these decomposition groups are equal to their respective inertia groups and we have
$I_{x,\mathcal{X}}=I_{r_{\mathcal{X}}(x),\Gamma}$.

Using this identification, the case where $x$ is an intersection point immediately follows.  We are thus left with the case of a generic ramification point $x$ of the morphism $\phi:X\rightarrow{Y}$.
Let $z:=r_{\mathcal{X}}(x)$. 
For any subgroup $H$ of $G$, we let $z_{H}$ be the image of $z$ under the natural map $X\rightarrow{X/H}$. We show $I_{x,\mathcal{X}}=I_{z,\mathcal{X}}$. By our earlier considerations, we then see $I_{x,\mathcal{X}}=I_{z,\Gamma}$. 

If $\sigma\in{I_{x,\mathcal{X}}}$, then $\sigma\in{D_{x,\mathcal{X}}}$. Then $\sigma$ must fix $z$ as well, because otherwise there would be at least two points in the closure of $x$ lying above the special fiber. So $\sigma\in{D_{z,\mathcal{X}}}$. But by our earlier assumption on the residue field $k$, we have 
 $D_{z,\mathcal{X}}=I_{z,\mathcal{X}}$. This yields ${I_{x,\mathcal{X}}}\subseteq{I_{z,\mathcal{X}}}$.

For the other inclusion, we use the following criterion. 
Let $H$ be a subgroup of $G$. Let $x_{H}$ be the image of $x$ in $\mathcal{X}/H$. The induced map $\mathcal{X}/H\rightarrow{\mathcal{Y}}$ is \'{e}tale at $x_{H}$ if and only if $H\supseteq{I_{x}}$.
This is a consequence of the second part of Proposition \ref{ramprop} in Section \ref{disbransect}.

We now only need to show $\mathcal{X}/I_{x}\rightarrow{\mathcal{Y}}$ is unramified at $z_{I_{x}}$. Suppose it is ramified at $z_{I_{x}}$. Then $z_{G}$ is a branch point of the covering $\mathcal{X}/I_{x}\rightarrow{\mathcal{Y}}$. Since $\phi_{X}$ is disjointly branched, $z_{G}$ is in the smooth part of the special fiber. This implies $\mathcal{Y}$ is regular at $z_{G}$, so we can use \emph{purity of the branch locus} (See \cite[Tag 0BMB]{stacks-project}) on some open subset $U$ of $\mathcal{Y}$ containing $z_{G}$ to conclude there must be a generic branch point $P$ such that $z_{G}$ is in the closure of $P$. Indeed, purity of the branch locus tells us a point of codimension 1 has to be in the branch locus and this cannot be a vertical divisor by our second assumption for disjointly branched morphisms. 
We must have $P=x_{G}$ because the branch locus is disjoint. This contradicts the fact that the morphism $\mathcal{X}/I_{x}\rightarrow{\mathcal{Y}}$ is unramified above $x_{G}$ (it is the largest extension with this property), so we conclude $\mathcal{X}/I_{x}\rightarrow{\mathcal{Y}}$ is unramified at $z_{I_{x}}$. In other words, we have $I_{z,\mathcal{X}}=I_{x,\mathcal{X}}$, as desired.  
\end{proof}

We use this result in the tropicalization algorithm to relate the inertia groups on the two-dimensional scheme $\mathcal{X}$ to inertia groups on the special fiber. This allows us to calculate inertia groups without calculating any normalizations. 
This in turn tells us how many edges and vertices there are in the pre-image of any edge $e$ or vertex $v$ in the dual graph of the special fiber of $\mathcal{Y}$. 

\subsection{From algebraic geometry to metric graphs}\label{algmetgraph}
In this section, we study the transition from the algebraic Galois coverings (as in Subsection \ref{disbransect}) to Galois coverings of metric graphs (as in Section \ref{galoiscovers}). To do this, we modify our graphs to reflect the geometry further by assigning lengths to the edges, adding weights to the vertices to account for genera, and adding leaves to account for the ramification coming from generic ramification points.

We start with the quotient map of \emph{graphs} $\Sigma(\mathcal{X})\rightarrow{\Sigma(\mathcal{Y})}$ coming from a disjointly branched morphism. 
We assign lengths to these edges by studying their local algebraic structure. The edges in $\Sigma$ correspond to ordinary double points on $\mathcal{X}$ and $\mathcal{Y}$. For any such point $x$ in $\mathcal{X}$,
$$
\widehat{\mathcal{O}}_{\mathcal{X},x}\simeq{\mathcal{O}_{K}[[u,v]]/(uv-\pi^{a})},
$$
where the completion is with respect to the maximal ideal $\mathfrak{m}_{x}$ of the local ring $\mathcal{O}_{\mathcal{X},x}$.
We define the length of the edge $l(e)$ corresponding to $x$ to be $l(e):=a$. Each vertex $v\in \Sigma$ corresponds to an irreducible component $\Gamma_v$ of genus $g(\Gamma_v)$ in the special fiber. We weight the vertices $v\in \Sigma$ by $w(v):=g(\Gamma_v)$.

We now add infinite leaf edges to the graphs $\Sigma(\mathcal{X})$ and $\Sigma(\mathcal{Y})$. We first perform a base change to make all the ramification points rational over $K$. Consider the morphism of smooth curves $\phi:X\rightarrow{Y}$. Let $P\in{X({K})}$ be a ramification point and let $Q=\phi(P)\in{Y({K})}$ be the corresponding branch point. The points $P$ and $Q$ reduce to exactly one component on $\mathcal{X}$ and $\mathcal{Y}$ respectively. We add a leaf edge $E_{P}$ to the vertex $V_{P}$ that $P$ reduces to and a leaf edge $E_{Q}$ to the vertex $V_{Q}$ that $Q$ reduces to. Doing this for every ramification point gives two loopless models $\tilde{\Sigma}(\mathcal{X})$ and $\tilde{\Sigma}(\mathcal{Y})$. There is a natural map between the two, which is induced by the map $\Sigma(\mathcal{X})\rightarrow{\Sigma(\mathcal{Y})}$ and sends leaf edges $E_{P}$ to $E_{Q}$. The integer $l'(E_{Q})/l(E_{P})$ we assign to these edges is $|I_{P}|$. There is a natural action of $G$ on this loopless model, given as follows. On $\Sigma(\mathcal{X})$, this is the usual action. For leaf edges $E_{P}$, we define an action by $\sigma(E_{P})=E_{\sigma(P)}$. This in accordance with the algebraic data, since $\sigma(P)$ reduces to $\sigma(V_{P})$. 
\begin{lemma}\label{connectram}
Let $\mathcal{X}\rightarrow{\mathcal{Y}}$ be a disjointly branched morphism as in Definition \ref{disbran} and let $\Sigma(\mathcal{X})\rightarrow{\Sigma(\mathcal{Y})}$ be the induced covering of intersection graphs from Theorem \ref{tropabeliantheorem}. For every edge $e\in\tilde{\Sigma}(\mathcal{X})$ (as defined above) corresponding to a point $x\in\mathcal{X}$, we have  $s_e=|I_{x}|$.
\end{lemma}
\begin{proof}
For edges corresponding to generic ramification points, this is by definition. For edges corresponding to intersection points, this follows from \cite[Chapter 10, Proposition 3.48]{liu2}.
\end{proof}
\begin{prop}\label{disbranprop1}
Let $\mathcal{X}\rightarrow{\mathcal{Y}}$ be a disjointly branched morphism as in Definition \ref{disbran}. The natural quotient map $\Sigma(\mathcal{X})\rightarrow{\Sigma(\mathcal{Y})}=\Sigma(\mathcal{X})/G$ from Theorem \ref{tropabeliantheorem} then yields a Galois covering of metric graphs $\tilde{\Sigma}(\mathcal{X})\rightarrow{\tilde{\Sigma}(\mathcal{Y})}$. 
\end{prop}
\begin{proof}
By construction every edge $e$ in $\tilde{\Sigma}(\mathcal{X})$ corresponds to either a generic (geometric) ramification point or an intersection point. To every point $x\in\mathcal{X}$, we then apply the Orbit-Stabilizer Theorem from group theory to obtain $|G|=|D_{P}|\cdot{\#\{\sigma(P):\sigma\in{G}\}}$. For generic ramification points and intersection points, we have  $I_{P}=D_{P}$, by our assumptions in Section \ref{disbransect}.  
We thus find the degree is just $|G|$ everywhere, so in particular it is independent of the edge $e$.

We now have to check the local Riemann-Hurwitz conditions at every vertex $v$ mapping to a vertex $w$. By Lemma \ref{connectram}, the quantities $s_e$ are just the ramification indices of the generic points reducing to $v$ and the indices of the edges. But these indices correspond with the indices on the special fiber by Proposition~\ref{ramind2}. Furthermore, the ramification points of the morphism on the special fiber $\Gamma\rightarrow{\Gamma'}$ (corresponding to $v\mapsto{w}$) all arise from either the closure of a generic ramification point (using purity of the branch locus as before) or as an intersection point of $\mathcal{X}$. These are all accounted for, so the Riemann-Hurwitz conditions must be satisfied.
\end{proof}

\begin{rem}\label{Admissiblecovering}
The notion of a disjointly branched morphism can be linked to that of an {\it{admissible covering}} (see \cite{admcov}) as follows. First, we note that the notion introduced in \cite[Definition 8]{admcov} should be a relative notion. Let $\mathcal{X}/\mathrm{Spec}(R)$ and $\mathcal{Y}/\mathrm{Spec}(R)$ be proper $n$-pointed semistable curves over the discrete valuation ring $R$, as defined in \cite[Definition 1.6.1]{Temkin2018}. We then define a {\it{relative admissible covering}} $\phi:\mathcal{X}\rightarrow{\mathcal{Y}}$ to be a finite morphism over $\mathrm{Spec}(R)$ such that the four conditions of \cite[Definition 8]{admcov} hold.

Suppose now that $\phi:\mathcal{X}\rightarrow{\mathcal{Y}}$ is a disjointly branched morphism. We mark these curves by the closures of the ramification and branch points respectively. By Proposition \ref{ramind2} and the second condition of a disjointly branched morphism, one sees that that $\phi$ is \'{e}tale outside the marked points and the singular points. We then immediately see that the first three conditions in \cite[Definition 8]{admcov} hold. For the last condition, one uses \cite[Lemma 5.5.5]{tropabelian} to see that the inertia groups are cyclic, after which the desired form in the fourth condition follows by a simple calculation on normalizations of local rings. We conclude that $\phi$ is a relatively admissible covering over $\mathrm{Spec}(R)$. 

The reason we chose to work with disjointly branched morphisms is as follows. This notion gives a concrete way of going from (Galois) covers of algebraic curves over a field to covers of strongly semistable models over a discrete valuation ring. Indeed, let $X\rightarrow{Y}$ be such a covering and let $\mathcal{Y}$ be any strongly semistable model. By a sequence of blow-ups, one then obtains a strongly semistable $\mathcal{Y}'$ in which the closure of the branch locus consists of smooth, disjoint sections over $R$. The normalization of $\mathcal{Y}'$ in the function field $K(X)$ might then be ramified at the codimension one primes in the special fiber. After a finite totally ramified extension $K\subset{K'}$ of the base field, this ramification disappears and the normalization of $\mathcal{Y}'\times{\mathrm{Spec}(R')}$ in the function field $K'(X)$ is then strongly semistable. Furthermore, the corresponding induced morphism $\mathcal{X}'\rightarrow{\mathcal{Y}'\times{\mathrm{Spec}(R')}}$ is a disjointly branched morphism.

We see that disjointly branched morphisms give a way of moving from tame algebraic coverings of curves over a field to tame coverings of (strongly) semistable models and thus to admissible coverings. In general, this step of moving from algebraic coverings to coverings of semistable models is quite hard due to issues of wild ramification. Understanding the combinatorics of this for superelliptic coverings would be an interesting direction of future research.

\end{rem}

\section{Tropicalization Algorithm}
\label{tropalgo}

In this section, we provide the algorithm producing the Berkovich skeleton for a superelliptic curve $X$ with superelliptic covering $X\rightarrow{\mathbb{P}^{1}}$. To do this, we first provide a semistable model $\mathcal{Y}$ of $\mathbb{P}^{1}$ separating the branch locus over $R$. The idea is to use Proposition \ref{ramind2} to reduce all the calculations to one dimensional schemes.
\subsection{A separating semistable model}\label{sepmodel}
We describe a semistable model $\mathcal{Y}$ of $\mathbb{P}^1$ separating the branch locus $B$ of $\phi:X\rightarrow{\mathbb{P}^{1}}$. 
We start with the model $\mathcal{Y}_{0}:=\text{Proj}R[X,Y]$ where $R[x,y]$ has the usual grading. The reader can think of this as being obtained from gluing together the rings $R[x]$ and $R[1/x]$. 
We now have a canonical reduction map $r_{\mathcal{Y}_{0}}$, which takes a closed point $P\in{\mathbb{P}_{K}^{1}}$ and sends it to the unique point in the closure of $P$ lying over the special fiber $(\mathcal{Y}_{0})_{s}=\mathbb{P}_{k}^{1}$ as in \cite[Section 10.1.3, Page 467]{liu2} or Section \ref{compinert}.
Informally, this map is given as reducing modulo the maximal ideal of $R$. This reduction map depends on the choice of the model $\mathcal{Y}_{0}$.

We now use this reduction map on the branch points $B$ to obtain a collection of points in the special fiber. We group together all points having the same reduction under this reduction map. This provides a subdivision of $B$ into subsets $B_{i}$. We consider the subsets with $|B_{i}|>1$. For these subsets with their corresponding reduced points $p_{i}$ we now blow-up the model $\mathcal{Y}_{0}$ in the points $p_{i}$. This gives a new model $\mathcal{Y}_{1}$. 
On this new model $\mathcal{Y}_{1}$, we again have a canonical reduction map $r_{\mathcal{Y}_{1}}$ and similarly consider the image of every subset $B_{i}$ under this reduction map to obtain a new subdivision $B_{i,j}$. For every two points $P_{1}$ and $P_{2}$ in $B$, we have they are in the same $B_{i,j}$ if and only if their reductions in $\mathcal{Y}_{1}$ are the same. This gives a new set of points $p_{i,j}$ (the reduction of the points in $B_{i,j}$) in the special fiber of $\mathcal{Y}_{1}$. We consider the points $p_{i,j}$ such that $|B_{i,j}|>1$. Blowing up these points $p_{i,j}$ gives a new model $\mathcal{Y}_{2}$.

 This process terminates: at some point all the $B_{i_{0},i_{1},...,i_{k}}$ have cardinality $1$, since the coordinates of the points on the special fiber of the blow-up are exactly the coefficients of the $\pi$-adic expansions of those points. 
The $\pi$-adic expansions of distinct $P_{i}$ and $P_{j}$ are different after a certain height $k$, giving different coordinates on the corresponding blow-up.
The last semistable model $\mathcal{Y}_{k}$ before the process above terminates is our separating semistable model. We simply call this model $\mathcal{Y}$. 

\subsection{Ramification indices for superelliptic coverings}
Let $A$ be a discrete valuation ring with 
maximal ideal $\mathfrak{m}_{A}=(\pi_{A})$, valuation $v_{A}$, field of fractions $K(A)$ and residue field $k(\mathfrak{m}_{A})$. Here we assume that $n$ is coprime to the characteristic of the residue field and that $A$ contains a primitive $n$-th root of unity. The discrete valuation rings in this section are the local rings of closed points on a curve and generic points of irreducible components of some semistable model $\mathcal{Y}$. 
\begin{exa}\label{RationalPoint}
Let $X\rightarrow{\mathbb{P}^{1}}$ be a superelliptic covering of degree $n$ over $K$ given by $(x,y)\mapsto{x}$ for the curve $X$ defined by
$
y^{n}=f(x),
$
where we can assume $f(x)$ is a polynomial in $K[x]$. For every nongeneric point $P_{\alpha}\in\mathbb{P}^{1}$, the local ring $\mathcal{O}_{\mathbb{P}^{1},P_{\alpha}}$ is a discrete valuation ring in the function field $K(x)$ and we write $v_{\alpha}$ for the corresponding valuation. 
If $P_{\alpha}$ is a $K$-rational point not equal to infinity, then this valuation can be calculated as follows. The point $P_{\alpha}$ corresponds to a maximal ideal of the form $(x-\alpha)$ for some $\alpha\in{K}$. Every $h\in{K(x)}$ can then be written as 
$$
h=(x-\alpha)^{k}\dfrac{g_{1}}{g_{2}},
$$
where $k\in\mathbb{Z}$ and the $g_{i}$ are polynomials in $K[x]$ such that $(x-\alpha)\nmid{g_{i}}$. We then have $v_{\alpha}(h)=k$. 

If $P_{\alpha}=\infty$, we proceed as follows. We write any $f$ as a quotient $f=g_{1}/g_{2}$ of two polynomials $g_{1},g_{2}$. We then have
$$
v_{\alpha}(f)=\mathrm{deg}(g_{2})-\mathrm{deg}(g_{1}).
$$
See \cite{Sticht2009} for more background on these valuations of function fields.
\end{exa}
We now consider an abelian extension $L$ of $K(A)$ of degree $n$. By Kummer theory, this extension can be represented as $$L=K(A)[y]/(y^{n}-f)$$ for some $f\in{K(A)}$.
We now give a criterion that allows us to calculate the orders of the inertia and decomposition groups for these extensions.  
We will use this criterion in Algorithm \ref{tropalgor} to determine the decomposition groups of the edges and vertices. 
\begin{lemma}
\label{SplittingLemma1} Let $A$ be a discrete valuation ring with
maximal ideal $\mathfrak{m}_{A}=(\pi_{A})$, field of fractions $K(A)$, valuation $v_{A}$ and residue field $k(\mathfrak{m}_{A})$. Assume that $n$ is coprime to the characteristic of the residue field. Let $L:=K(A)[y]/(y^{n}-f)$ be a finite abelian extension of $K(A)$ of degree $n$ with $f\in{K(A)}$ and write $f=\pi^{k}_{A}u$, where $k=v_{A}(f)$ and $u$ is a unit in $A$. 
Then the following are true.

\begin{enumerate}
    \item Let $\mathfrak{m}_{B}$ be a maximal ideal lying above $\mathfrak{m}_{A}$.  
Then $$|I_{\mathfrak{m}_{B}}|=\dfrac{n}{\mathrm{gcd}(n,k)}.
$$
\item Suppose that $k$ is divisible by $n$ so that the extension $\mathfrak{m}_{B}\supset{\mathfrak{m}_{A}}$ is unramified. 
Let $r$ be the largest divisor of $n$ such that $\overline{u}=s^{r}$ for some $s\in{k(\mathfrak{m}_{A})}$. Then
$$
|D_{\mathfrak{m}_{B}}|=\dfrac{n}{r}.
$$ 
\end{enumerate}

\end{lemma}

\begin{proof}
We write $w_{B}$ for the valuation corresponding to $\mathfrak{m}_{B}$. To calculate the order of the inertia and decomposition groups, we can work over the completion $\hat{K}_{\mathfrak{m}_{A}}$ of $K$ with respect to $v_{A}$. Indeed, the order of the inertia group is the ramification index $e_{\mathfrak{m}_{B}/\mathfrak{m}_{A}}$, which is equal to the index $[w_{B}(\hat{L}_{\mathfrak{m}_{B}}^{*}):v_{A}(\hat{K}^{*}_{\mathfrak{m}_{A}})]$ of value groups over the completion. The order of the decomposition group in turn is equal to the product $e_{\mathfrak{m}_{B}/\mathfrak{m}_{A}}\cdot{}f_{\mathfrak{m}_{B}/\mathfrak{m}_{A}}$, where $f_{\mathfrak{m}_{B}/\mathfrak{m}_{A}}$ is the degree of the residue field extension corresponding to $\hat{K}_{\mathfrak{m}_{A}}\subset{}\hat{L}_{\mathfrak{m}_{B}}$. Note that the extension $\hat{K}_{\mathfrak{m}_{A}}\subset{\hat{L}_{\mathfrak{m}_{B}}}$ is given by adjoining a root $\gamma$ of $y^{n}-f$ to $\hat{K}_{\mathfrak{m}_{A}}$, see \cite[Section 3]{Ste2}. 
We will write $K$ and $L$ for $\hat{K}_{\mathfrak{m}_{A}}$ and $\hat{L}_{\mathfrak{m}_{B}}$ respectively. Let $n':=n/\mathrm{gcd}(n,k)$ and $k':=k/\mathrm{gcd}(n,k)$. We write $\alpha:=\pi_{A}^{1/n'}$ for an $n'$-th root of $\pi_{A}$ in the algebraic closure of $K$. The extension 
\begin{equation}
    K\subset{K(\alpha)}=:K'
\end{equation}
is then totally ramified of degree $n'$. 
We claim that $\gamma$ is in the maximal unramified extension 
$K'^{\mathrm{unr}}$ of $K'$. Indeed, consider the element
\begin{equation}
    \gamma':=\dfrac{\gamma}{\alpha^{k'}}.
\end{equation}
Then $\gamma'^{n}=u$. But the polynomial $y^{n}-u$ defines an unramified extension of $K$ (and thus of $K'$) by our assumption on $n$, so $\gamma'$ is defined over $K^{\mathrm{unr}}\subset{K'^{\mathrm{unr}}}$. This then 
implies that $\gamma=\gamma'\cdot{\alpha^{k'}}$ 
is defined over the maximal unramified extension of $K'$, as desired. We now have the chain of finite extensions 
\begin{equation}
    K^{\mathrm{unr}}\subset{K^{\mathrm{unr}}(\gamma)}\subset{K^{\mathrm{unr}}(\alpha)}.
\end{equation}
Since $\mathrm{gcd}(n',k')=1$, we can find integers $h_{i}$ such that $h_{1}k'+h_{2}n'=1$. Since $\alpha^{k'}\in{K^{\mathrm{unr}}(\gamma)}$, we find  $\alpha/\pi_{A}^{h_{2}}=\alpha^{h_{1}k'}\in{K^{\mathrm{unr}}(\gamma)}$ and thus $\alpha\in{K^{\mathrm{unr}}(\gamma)}$. The equality $K^{\mathrm{unr}}(\gamma)={K^{\mathrm{unr}}(\alpha)}$ then directly implies the formula for the inertia group.  

For the decomposition group formula, let $\gamma$ be as above. By considering $\gamma/\pi_{A}^{h}$ instead of $\gamma$ for a suitable integer $h$, we can assume that $k=0$. We have to find the degree of the field extension $K\subset{K(\gamma)}$. Let $\beta\in{K}$ be an element that reduces to $s$. We claim that 
\begin{equation}
    y^{n/r}-\beta
\end{equation}
is irreducible over $K$. To see this, it suffices to show that the corresponding polynomial is irreducible over $k(\mathfrak{m}_{A})$. Suppose that it is reducible. Using our assumption that $k(\mathfrak{m}_{A})$ contains a square root of $-1$ if $n$ is divisible by $4$ together with \cite[Theorem 1.6]{karpilovsky1989topics}, we find thatthere exists a prime $q$ dividing $n/r$ with $s\in{k(\mathfrak{m}_{A})^{q}}$. But then $r$ is not the largest divisor of $n$ with $s^{r}=\overline{u}$, a contradiction. We conclude that $y^{n/r}-\beta$ is irreducible over $K$. We write $\gamma'$ for a root of this polynomial over an algebraic closure of $K$. We then have that 
\begin{equation}
    (\gamma/\gamma')^{n}=u'
\end{equation}
for a unit $u'$ in $A$ with reduction $1$. By Hensel's lemma, we thus have that $K(\gamma)=K(\gamma')$, which then also yields the formula for the decomposition group, as desired.   
\end{proof}

Let $\phi:X\rightarrow{\mathbb{P}^{1}}$ be a superelliptic covering and consider the canonical model $\mathcal{Y}$ constructed in Section \ref{sepmodel}. We do not need to write the equations for this model, and we may instead work with the intersection graph, which is the tropical separating tree of these points minus the leaves at the end. For this canonical model $\mathcal{Y}$, we take the corresponding disjointly branched morphism $\mathcal{X}\rightarrow{\mathcal{Y}}$ obtained by normalization after a finite extension. That is, we take a finite extension of $K$ to eliminate the ramification on the components of the special fiber of $\mathcal{Y}$ and then we take the normalization $\mathcal{X}$ of $\mathcal{Y}$ inside the function field $K(X)$ of $X$. By \cite[Theorem 2.3]{liu_lorenzini_1999} and \cite[Proposition 4.1.1]{tropabelian}, the morphism $\phi_{\mathcal{X}}:\mathcal{X}\rightarrow{\mathcal{Y}}$ is then disjointly branched, as defined in Definition \ref{disbran}. We use this disjointly branched morphism $\phi_{\mathcal{X}}$ throughout this section. 

\begin{prop}
\label{edgeprop}
Let $P_{\alpha}\in\mathbb{P}^{1}({K})$ be a (generic) branch point of the superelliptic covering $\phi:X\rightarrow{\mathbb{P}^{1}}$ 
given by the equation $y^n=f(x)$ with a corresponding superelliptic morphism of metric graphs $\phi_{\Sigma}:X_{\Sigma}\rightarrow{T}$ induced from disjointly branched morphism  $\mathcal{X}\rightarrow{\mathcal{Y}}$. Let $c_{\alpha}:=v_{\alpha}(f(x))$, where $v_{\alpha}$ is the valuation associated to $P_{\alpha}$ in the function field ${K}(x)$. For any point $x\in{\mathcal{X}}$, we let $I_{x}$ and $D_{x}$ be the inertia group and decomposition group of $x$, as defined in Section \ref{galoiscovers2}.
\begin{enumerate}
\item Let $Q_{\alpha}$ be any point in the preimage of $P_{\alpha}$, and let $\tilde{Q}_{\alpha}:=r_{\mathcal{X}}(Q_{\alpha})$. 
Then
\begin{equation*}
|I_{Q_{\alpha}}|=n/\gcd(c_{{\alpha}},n)=|I_{\tilde{Q}_{\alpha}}|.
\end{equation*}
\item \label{edgeprop2}Let $\psi$ be a rational function on $T$ satisfying $\Delta(\psi) = \rho(\text{div}(f))$. Let $\psi_e$ be the slope of $\psi$ along the edge $e$ of $T$. Let $e'$ be any edge lying above $e$. For any edge $e'$ lying above $e$,
\begin{equation*}
|I_{e'}|=n/\gcd(\psi_{e},n).
\end{equation*}
In other words, there are $\text{gcd}(\psi_{e},n)$ edges lying above $e$. 
\item \label{edgeprop3} Let $g_v$ be the number of vertices in $\Sigma$ lying above $v \in T$ and let $m_{e}=\mathrm{gcd}(\psi_{e},n)$. Then

\begin{equation*}
g_{v}={\mathrm{gcd}(m_{e})}=\mathrm{lcm}(n/m_{e}),
\end{equation*}
where the least common multiple and greatest common divisor are taken over edges $e$ adjacent to $v$.

\end{enumerate}
\end{prop}
\begin{proof} $\ $

\begin{enumerate}
\item 
We have $|I_{Q_{\alpha}}|=\dfrac{n}{\mathrm{gcd}(c_{{\alpha}},n)}$ by applying the second part of Lemma \ref{SplittingLemma1} to the extension $\mathcal{O}_{\mathbb{P}^{1},P_{\alpha}}\rightarrow{\mathcal{O}_{X,Q_{\alpha}}}$ of discrete valuation rings arising from $X\rightarrow{\mathbb{P}^{1}}$. Using Proposition \ref{ramind2}, we conclude that $|I_{Q_{\alpha}}|=|I_{\tilde{Q}_{\alpha}}|$ .

\item For the second statement, pick any vertex $v$ with corresponding irreducible component $\Gamma$ containing the edge $e$. We consider the $\Gamma$-modified form of $f$, defined as follows. The component $\Gamma$ has a generic point $y$ with discrete valuation ring $\mathcal{O}_{\mathcal{Y},y}$, valuation $v_{\Gamma}$, and uniformizer $\pi$. We set $k:=v_{\Gamma}(f)$ and define the $\Gamma$-modified form to be the element
\begin{equation*}
f^{\Gamma}:=\dfrac{f}{\pi^{k}}.
\end{equation*}
The corresponding morphism of components is described by $(y')^{n}-f^{\Gamma}$, where $y'=\dfrac{y}{\pi^{k/n}}$. On the special fiber, the intersection point corresponds to a smooth point of $\Gamma$. 
By the Poincar\'{e}-Lelong formula, as presented in \cite[Theorem 3.3.2.]{tropabelian}, 
the valuation of $f^{\Gamma}$ at this smooth point is exactly the slope of the function $\psi$ on $e$.
As in the previous statement, the ramification index on the special fiber is $n/\text{gcd}(\psi_{e},n)$. By Proposition \ref{ramind2}, this is the order of the inertia group at $e'$, as desired.
\item For the third statement, we will apply Lemma \ref{SplittingLemma1}, part (iii) to the extension of discrete valuation rings 
\begin{equation}
\mathcal{O}_{\mathcal{Y},y_{v}}\rightarrow{\mathcal{O}_{\mathcal{X},y_{v'}}},
\end{equation}
where $y_{v}$ and $y_{v'}$ are the generic points corresponding to the vertices $v$ and $v'$ respectively. Note that the $\Gamma$-modified form defined in the previous part is exactly the unit $u$ studied in Lemma \ref{SplittingLemma1}. We thus have to find the largest $r$ such that $\overline{f^{\Gamma}}=s^{r}$ for some $s\in{k(y_{v})}=k(\mathbb{P}^{1}_{k})$. Here we identified the residue field of $y_{v}$ with the function field of the projective line over $k$. Since $\mathbb{P}^{1}_{k}$ has trivial Picard group, a rational function $h$ is a $j$-th power of some other rational function if and only if the valuations of $h$ at all points of $\mathbb{P}^{1}$ are divisible by $j$. We thus have to find the greatest common divisor of the valuations of $\overline{f^{\Gamma}}$ and $n$. The points that nontrivially contribute to the valuations of $\overline{f^{\Gamma}}$ correspond exactly to the edges adjacent to $v$. Since these valuations are equal to the slope $\psi_{e}$ of $\psi$ along $e$, we see that we obtain the statement of the proposition.
\end{enumerate}
\end{proof}

\subsection{The algorithm}
We now give an algorithm producing the Berkovich skeleton of a curve $X$ defined by an equation $y^{n}=f(x)$ for some $n\geq{2}$ and $f(x)\in{K(x)}$.
This algorithm generalizes the known algorithm for finding the Berkovich skeleton of hyperelliptic curves from \cite[Section 2]{bbc}.
We take as input to the algorithm a polynomial $f(x) \in K[x]$, which we may do because for $f(x)$ of the form $f(x)=g(x)/h(x)$, we may multiply both sides of $y^n=f(x)$ by $h(x)^{n}$ and make a change of coordinates $\tilde{y}=h(x)\cdot{y}$ to obtain the integral equation $(\tilde{y})^n=g(x)h(x)^{n-1}$.

\begin{algo}[Tropicalization Algorithm] $\ $
\label{tropalgor} 
$\ $\\

\textit{Input}: A curve $X$ defined by the equation $y^n = f(x)$ over a field with a valuation $v$.

\textit{Output}: The Berkovich skeleton $X_\Sigma$ of $X$.
\begin{enumerate}
\item \textit{Compute finite expansions of the roots of $f$.}
If $K$ is the field of Laurent series over $\mathbb{C}$, then the Newton-Puiseux Method \cite{walker} can compute finite expansions for the roots of the polynomial $f$. An explicit upper bound for the needed height of this expansion is given by $v(\Delta(f))$, where $\Delta(f)$ is the discriminant of $f$. Typically, this method is offered as a proof that the field of Puiseux series is algebraically closed, but it can also be used to actually find the roots of univariate polynomials over the Puiseux series. This method has been implemented in \texttt{Maple}\cite[\texttt{algcurves}]{Maple} and \texttt{Magma} \cite{magma}.
\item \textit{Compute the tree $T$}. This is the abstract tropicalization of $\mathbb{P}^1$ together with the marked ramification points $Q_1, \ldots, Q_s$. This is done in the following way (See \cite[Section 4.3]{tropicalbook}).
\begin{enumerate}
\item Let $M$ be the $2 \times s$ matrix whose columns are the branch points $Q_1, \ldots, Q_s$. Let $m_{ij}$ be the $(i,j)$-th minor of this matrix.
\item Let $d_{ij} = N- 2v(m_{ij})$, where $v$ is the valuation on $K$ and $N$ is an integer such that $d_{ij}\geq 0$. 
\item The number $d_{ij}$ is the distance between leaf $i$ and leaf $j$ in the tree $T$. These distances uniquely specify the tree $T$, and one can use the Neighbor Joining Algorithm \cite[Algorithm 2.41]{PS05} to reconstruct the tree $T$ from these distances.
\end{enumerate}
\item \textit{Compute the slopes $\psi_e$ along each edge of $T$}. The divisor $\rho(\text{div}(f))$ is a principal divisor on $T$, and so there exists a rational function $\psi$ on $T$ with $\Delta(\psi) = \rho(\text{div}(f))$ (as defined in \cite[Page 4]{baker}). One can compute $\rho(\text{div}(f))$ by observing where the zeros and poles of $f$ specialize. Use this to compute the slopes $\psi_e$ of $\psi$ along edges $e$ of $T$ using Lemma \ref{slopelemma}.
\item \textit{Compute the intersection graph of $\mathcal{X}_s$.}
\begin{enumerate}
\item \textit{Edges.} The number of preimages of each edge is $\text{gcd}(\psi_e,n)$ by Proposition \ref{edgeprop}.\ref{edgeprop2}. 
\item \textit{Vertices.} The number of preimages of each vertex $v$ is $\text{gcd}( n,\psi_e |e \ni v)$ by Proposition~\ref{edgeprop}.\ref{edgeprop3}.
\end{enumerate}
\item \textit{Determine the edge lengths and vertex weights to find $X_\Sigma$}. 
\begin{enumerate}
\item \textit{Edges}. If an edge $e$ in $T$ has length $l(e)$, then the length of each of its preimages in $X_\Sigma$ is $\frac{l(e)\cdot \text{gcd}(\psi_e,n)}{n}$, by Proposition \ref{edgeprop} and \cite[Chapter 10, Proposition 3.48, Page 526]{liu2}. Remove any infinite leaf edges.
\item \textit{Vertices}. The weight on each vertex $v'$ is determined by the local Riemann-Hurwitz formula. The degree $d$ at $v'$ is $n/g_{v}$, where $g_{v}$ is the number of vertices above $v$. Here we use the fact that the order of the decomposition group of the generic point corresponding to a vertex lying above $v$ is equal to the degree.  
The weight of $v'$ is determined by
\begin{equation*}
2w(v')-2=-2\cdot{d}+\sum_{e'\ni v'}\left( \frac{n}{\text{gcd}(n,\psi_e)} -1\right).
\end{equation*}
\end{enumerate}
\end{enumerate}
\end{algo}

\begin{theorem}[Tropicalization Algorithm]
\label{tropalgothm}
The Tropicalization Algorithm \ref{tropalgor} terminates and is correct.
\end{theorem}

\begin{proof} The tree $T$ created in the algorithm is the tree obtained from the canonical semistable model in Section \ref{sepmodel} with the leaves attached. The formulas for the number of preimages of the edges and the vertices are given by Proposition \ref{edgeprop}, parts \ref{edgeprop2} and \ref{edgeprop3} respectively. There is only one graph up to a labeling of the vertices satisfying the covering data found in the algorithm. Indeed, the covering data naturally give a 2-cocycle (in terms of \v{C}ech cohomology) on $T$, which must be trivial. We thus obtain the intersection graph of the semistable model $\mathcal{C}$. Contracting any leaf edges yields the Berkovich skeleton. 
\end{proof}

\begin{exa}
\label{k33}
We compute the abstract tropicalization of the curve defined by the equation
$$y ^3 = x^2(x - \pi)(x - 1)^2(x -1- \pi)(x - 2)^2(x - 2-\pi).$$
\begin{enumerate}
\item The matrix $M$ is
$$
M =
\begin{bmatrix}
0 & \pi & 1 & 1+\pi &2 & 2 + \pi \\
1 & 1 & 1 & 1 & 1 &1
\end{bmatrix},
$$
and so the vector $m$ (organized lexicographically) is
\begin{align*}
m=
(&
-\pi, \ -1, \ -1-\pi, \  -2, \ -2-\pi, \ \pi-1, \ -1, \ \pi-2, \ -2,\ -\pi, \ -1,\\
&-1-\pi,-1+\pi,\ -1, \ -\pi 
).
\end{align*}
Taking $N = 2$, we have
$
m=
(0, \ 2, \ 2, \  2, \ 2, \ 2, \ 2, \ 2, \ 2,\ 0, \ 2,\ 2,\ 2,\ 2, \ 0 
).
$
Therefore, the tree is as displayed in Figure \ref{k33tree}.
\item We have
$\text{div}(f) = 2(0) + (\pi) + 2(1) + (\pi+1) + 2 (2) +(2 + \pi) -9(\infty).$
Then, 
$\rho(\text{div}(f)) = 3 v_{12} +3 v_{34} + 3 v_{56} -9v,$ and
$\psi_{e_{12}}=\psi_{e_{34}}=\psi_{e_{56}}=3.$ On all leaf edges $\psi_e$ is 1 or 2.
\item Each of the edges $e_{12},e_{34},e_{56}$ has 3 preimages, and all leaves have 1 preimage. We can contract these in the tropical curve, so we do not draw them in the graph, but we mention them here because they are necessary for bookkeeping the ramification in the formulas. The middle vertex $v$ has 3 preimages, and the other vertices have 1 preimage. So, the graph is $K_{3,3}$.
\item The lengths of all interior edges in the tree $T$ were 1. These lengths are preserved in $K_{3,3}$ because all edges were unramified. The weights on all vertices are 0. For example,
$$
w(v_{12}) = -3 + 1 +(3(3/3-1)+2(3/1)-1)/2 = 0.
$$
So, the abstract tropicalization of our curve is the metric graph in Figure \ref{k33graph}. Each vertex is labeled with its image in the tree $T$.
\begin{figure}[ht]
\centering
\begin{minipage}{.4\textwidth}
 \centering
 \includegraphics[height=1.3in]{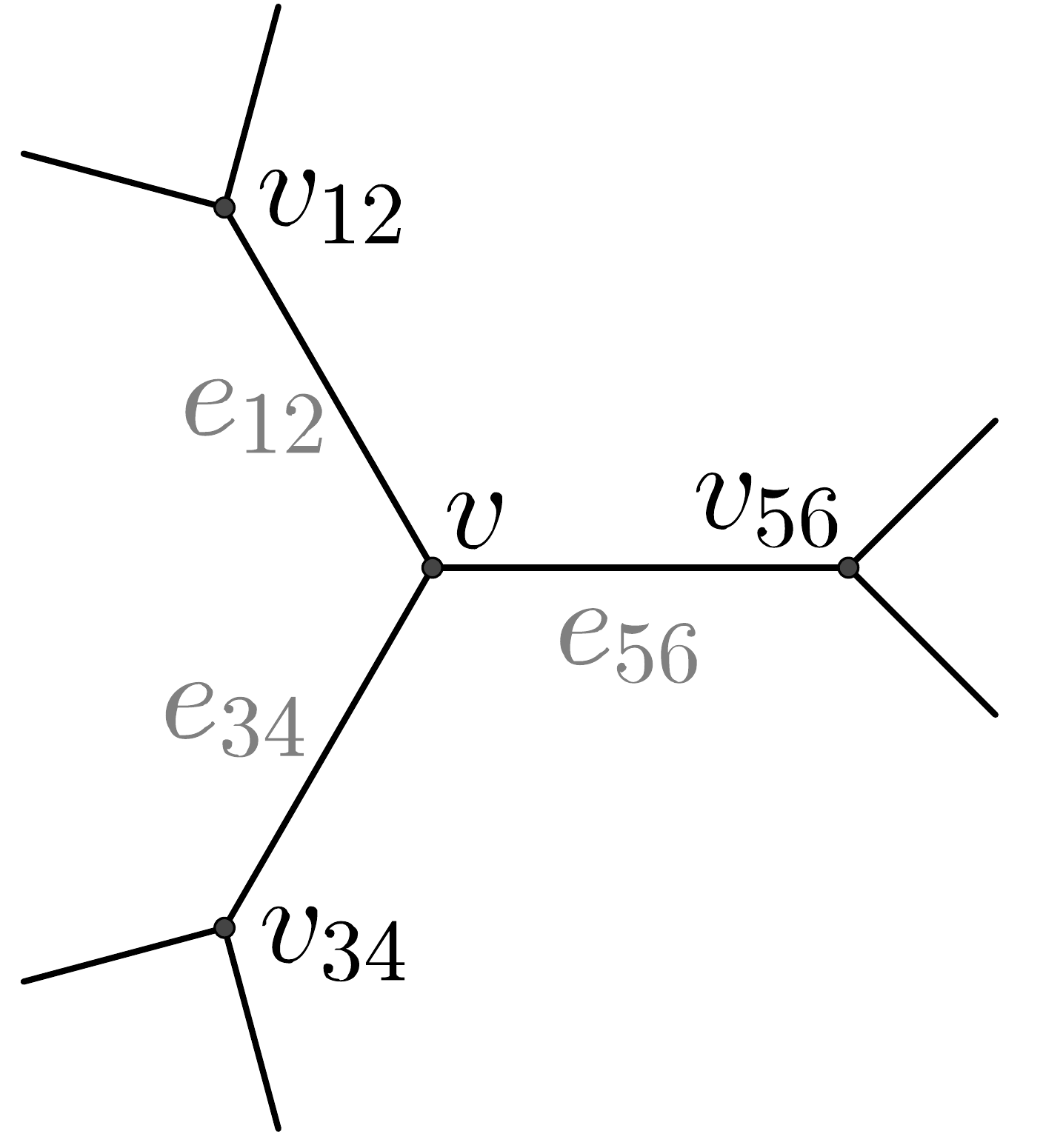}
 \captionof{figure}{The tree $T$ in Example \ref{k33}.}
 \label{k33tree}
\end{minipage}%
\hspace{0.3 in}
\begin{minipage}{.4\textwidth}
 \centering
 \includegraphics[height=1.15in]{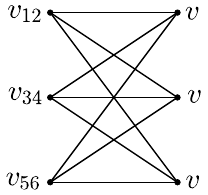}
 \captionof{figure}{Tropicalization of the curve in Example \ref{k33}.}
 \label{k33graph}
\end{minipage}
\end{figure}
\end{enumerate}
\end{exa}

\begin{exa}
In \cite{shimura2} the author shows there is a unique Shimura-Teichm\"{u}ller curve of genus three, $X_3$, defined by the equation $y^4 = x(x-1)(x-\pi)$, and there is a unique Shimura-Teichm\"{u}ller curve of genus four, $X_4$, defined by the equation $y^6 = x(x-1)(x-\pi).$
In \cite[Section 2]{shimura}, the authors compute the period matrix of $X_4$. We now use the Tropicalization Algorithm to compute their Berkovich skeleta.

\begin{enumerate}
\item
In both cases, the ramification points are $0,1,\pi$, and $\infty$. The corresponding tree is in Figure \ref{shimurasT}, where the interior edge has length 1. We call the interior vertices $v_1$ and $v_2$.
\item 
The divisor of $f:=x(x-1)(x+\pi)$ is 
$
\text{div}(f)=(0)+(\pi)+(-1)-3(\infty).
$
The divisor specializes to
$
\rho(\text{div}(f))=2v_{0}-2v_{1}.
$
The corresponding rational function $\psi$ has slope $2$ on the only edge in the tree.
\item We have $\text{gcd}(\psi_e,n)=2$ in both cases. Therefore, the edge has 2 preimages. Both vertices on the tree have leaves, so both $v_0$ and $v_1$ each have one preimage each in the graphs $X_{3,\Sigma}$ and $X_{4,\Sigma}$.
\item 

The length of the interior edge in the tree is 1, so in $X_3$ there are two edges of length $1/2$ and in $X_4$ there are two edges of length $1/3$.
 For the genera of the vertices, we apply the Riemann-Hurwitz formula to obtain the complete picture of the graphs; $X_{3,\Sigma}$ is in Figure \ref{fig:shimuras1} and $X_{4,\Sigma}$ is in Figure \ref{fig:shimuras2}.
 \begin{figure}[ht]
 \centering
\begin{minipage}{.38\textwidth}
 \centering
 \includegraphics[height=0.7in]{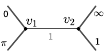}
 \captionof{figure}{\footnotesize{The base tree for the Shimura-Teichm\"{u}ller curves.}}
 \label{shimurasT}
\end{minipage}%
\hspace{-0.5 in}
\centering
\begin{minipage}{.38\textwidth}
 \centering
 \includegraphics[height=0.7in]{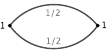}
 \captionof{figure}{\footnotesize{The tropical genus 3 Shimura-Teichm\"{u}ller curve.}}
 \label{fig:shimuras1}
\end{minipage}%
\hspace{-0.5 in}
\begin{minipage}{.38\textwidth}
 \centering
 \includegraphics[height=0.7in]{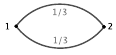}
 \captionof{figure}{\footnotesize{The tropical genus 4 Shimura-Teichm\"{u}ller curve.}}
 \label{fig:shimuras2}
\end{minipage}
\end{figure}
\end{enumerate}
\end{exa}

\section{Realizability}

\label{realizability}
In this section, we study realizability for superelliptic covers. We show every superelliptic cover of prime degree of metric graphs comes from an algebraic superelliptic cover.
A similar result was proved for degree $d$ admissible coverings in \cite[Theorem 1]{admcov} using techniques from \cite{ABBR1}: for every degree $d$ admissible covering of metric graphs $X_\Sigma \rightarrow T$, there exists an {\it{admissible}} covering $\mathcal{X}\rightarrow\mathcal{T}$ over $R$ tropicalizing to $X_\Sigma \rightarrow T$ (see Remark \ref{Admissiblecovering}). The proof uses a tame Berkovich-theoretic semistable lifting theorem \cite[Theorem 7.4]{ABBR1}. A careful examination of the proof then also shows that any Galois covering of metric graphs lifts to an admissible Galois covering of curves. In this sense, the realizability theorem we prove in this section can be considered as a special case of \cite[Theorem 7.4]{ABBR1} and \cite[Theorem 1]{admcov}. 
The main difference here is that the techniques used in \cite{ABBR1} are purely analytic and a priori only give coverings of analytic spaces. Using a GAGA-type result for Berkovich spaces, one then obtains a covering of algebraic curves. This generally makes it hard to write down the explicit equations for the curves, let alone the corresponding coverings. 
\begin{exa}
Consider the multiplicative group $G:=\overline{K}^{*}/\langle{q}\rangle$, where $q\in{\overline{K}}^{*}$ with $v(q)>0$. There is then a canonical way to associate an analytic space $X^{\mathrm{an}}/\overline{K}$ (in the Berkovich sense) to $G$ such that the type-$1$ points of $X^{\mathrm{an}}$ are given by $G$, see \cite[Section 9.2]{bosch2014} for the rigid-analytic construction. This curve $X^{\mathrm{an}}$ is also known as a {\it{Tate curve}}. This defines a proper, reduced, smooth, one-dimensional analytic space, which therefore must come from a unique proper, reduced, smooth algebraic curve $X/\overline{K}$ by \cite[Corollary 3.4.14]{berkovich2012}. This curve is an elliptic curve and an explicit equation for $X$ is given in \cite[Page 410, Theorem 1.1]{silv2} using the $a_{4}(q)$ and $a_{6}(q)$ modular invariants. 
\end{exa}

Unlike in \cite{ABBR1} and \cite{admcov}, our approach is constructive; the proof of our realizability theorem presents a method for finding the defining equation of a curve $X$. That is, we give a polynomial $f\in{K[x]}$ such that the covering $X\rightarrow{\mathbb{P}^{1}}$ given by the function field extension $K(x)\subset{K(x)[y]/(y^{p}-f(x))}$ tropicalizes to the desired superelliptic covering of metric graphs. 

We first recall the set up.
Given a superelliptic covering of curves $X \rightarrow \mathbb{P}^1$, we obtain a superelliptic covering of metric graphs $\psi:\Sigma \rightarrow T$ by computing the tree $T$ with model $T_{M}$ and the divisor $\rho(\text{div}(f)) = \sum a_i P_i$ on $T$ (we recall from before that this is the specialization of the divisor $\text{div}(f)$ to the tree $T_{M}$). The main difficulty in reversing this process is finding $a_i$ which give the graph $\Sigma$.
We show inductively there are enough ways of assigning values to the $a_{i}$ such that the desired tropicalization is obtained. To get acquainted with the sort of problem we are dealing with here, we treat an illustrative example.
\begin{exa}\label{RealizabilityExample1}
Let $T$ be the metric tree arising from the loopless model $T_{M}$ as in Figure \ref{FIG:harmonic_ex}. Here we assume that the edges $f_{1}$ and $f_{4}$ have length $1$.  Let $S$ be a fixed set of $6$ points in $\mathbb{P}^{1}(K)$ whose associated tropical tree is $T$ and let $\Sigma\rightarrow{T}$ be a superelliptic covering of $T$. We would like to find all algebraic superellipic coverings $X\rightarrow{\mathbb{P}^{1}}$ such that the induced covering from Proposition \ref{disbranprop1} is $\Sigma\rightarrow{T}$. Here we use the separating semistable model $\mathcal{Y}$ from Section \ref{sepmodel} in Proposition~\ref{disbranprop1}.

The degree $3$ superelliptic coverings of $\mathbb{P}^{1}$ that are ramified only above points in $S$ are given by rational functions $f\in{K(x)}$ such that $v_{P}(f)\not\equiv{0}\bmod{3}$ only for $P$ in $S$. Indeed, any such $f$ gives an extension of function fields
\begin{equation*}
K(x)\rightarrow{K(x)[y]/(y^3-f)}
\end{equation*}
with a superelliptic covering of smooth curves $X\rightarrow{\mathbb{P}^{1}}$ that is only ramified above points in $P$. Conversely, any degree $3$ superelliptic covering $X\rightarrow{\mathbb{P}^{1}}$ arises from a function field extension $K(x)\rightarrow{K(x)[y]/(y^{3}-f)}$ for some $f\in{K(x)}$ by Kummer theory. It then follows from Lemma \ref{SplittingLemma1} that $v_{P}(f)\not\equiv{0}\bmod{3}$ if and only if the covering is ramified above $P$.

We will now classify the coverings $\Sigma\rightarrow{T}$ arising from these $f$ in terms of the valuations $v_{P}(f)$ of $f$ at points $P\in{S}$. 

Let $S=(P_{5},P_{6},P_{7},P_{8},P_{9},P_{10})$ be a set of $6$ points such that the induced tropical tree from Section \ref{sepmodel} is equal to $T$. Here the point $P_{i}$ corresponds to the leaf $f_{i}$. For instance, we can take the projective points
$$
\{P_5,P_6,P_7,P_8,P_9,P_{10}\} = \{
[\pi^2:1],
[2\pi^2:1],
[\pi:1],
[1:1],
[2:1],
[1:0]\}.
$$

For any choice of integers $(c_{5},c_{6},c_{7},c_{8},c_{9},c_{10})$ with $c_{5}+c_{6}+c_{7}+c_{8}+c_{9}+c_{10}=0$, we canonically obtain an $f\in{K(x)}$ with $v_{P_{i}}(f)=c_{i}$. 
For example, if we do this for the points defined above and the list $(1,2,1,-1,2,-5)$, we obtain the rational function 
$$
f=(x-\pi^{2})(x-2\pi^{2})^{2}(x-\pi)(x-1)^{-1}(x-2)^{2}.
$$

Consider the curve defined by $y^3=f$ for any $f$ defined by a sequence $(c_{5},c_{6},c_{7},c_{8},c_{9},c_{10})$ as above. By normalizing, we assume that $v_{P}(f)=0$ for $P$ not in $S$. We will also assume that $v_{P}(f)\not\equiv{0}\bmod{3}$ for every $P\in{S}$. 

By Lemma \ref{slopelemma}, the slope of the piecewise linear function $\psi_{f}$ corresponding to $f$ (see Section \ref{specialdivisor}) on $f_{1}$ is equal to $c_{5}+c_{6}$. By Proposition \ref{edgeprop}, we see that the induced morphism of metric graphs from Proposition \ref{disbranprop1} is split above $f_{1}$ if and only if $$c_{5}+c_{6}\equiv{0\bmod{3}}.$$ The pairs $(\overline{c}_{5},\overline{c}_{6})$ (where the bar denotes the residue class in $\mathbb{F}_{3}$) that correspond to coverings that are split above $f_{1}$ are consequently given by $(\overline{1},\overline{2})$ and $(\overline{2},\overline{1})$. The pairs $(\overline{c}_{5},\overline{c}_{6})$ that correspond to coverings that are not split above $f_{1}$ are given by $(\overline{1},\overline{1})$ and $(\overline{2},\overline{2})$.

We now similarly consider the slope of the piecewise linear function $\psi_{f}$ on the edge $f_{4}$. Again by Lemma \ref{slopelemma}, we see that the induced morphism is split above $f_{4}$ if and only if $$c_{5}+c_{6}+c_{7}\equiv{0}\bmod{3}.$$ We can now express $\Sigma$ in terms of triples $(\overline{c}_{5},\overline{c}_{6},\overline{c}_{7})$. The triples that correspond to coverings that are split on $f_{1}$ and nonsplit on $f_{4}$ are given by
\begin{equation}\label{TriplesSplitNonSplit}
\{(\overline{2},\overline{1},\overline{1}),(\overline{2},\overline{1},\overline{2}),(\overline{1},\overline{2},\overline{1}),(\overline{1},\overline{2},\overline{2})\}.
\end{equation}The triples that correspond to coverings that are nonsplit on $f_{1}$ and split on $f_{4}$ are given by $$\{(\overline{1},\overline{1},\overline{1}),(\overline{2},\overline{2},\overline{2})\}$$ and the triples that correspond to coverings that are nonsplit on both edges are given by $$\{(\overline{1},\overline{1},\overline{2}),(\overline{2},\overline{2},\overline{1})\}.$$

Note that the covering $\Sigma\rightarrow{T}$ is independent of the valuations $c_{8},c_{9}$ and $c_{10}$ as soon as the others are determined. 

Let us now find a rational function $f\in{K(x)}$ that gives rise to the superelliptic covering in Figure \ref{FIG:harmonic_ex}. The covering is split on $f_{1}$ and nonsplit on $f_{4}$. We thus see that any of the four triples in Equation \ref{TriplesSplitNonSplit} gives rise to the desired superelliptic covering. To be explicit, take 
$$
f=(x-\pi^{2})^2(x-2\pi^{2})^1(x-\pi)^2(x-1)(x-2).
$$
The covering given by $(x,y)\mapsto{x}$ for $y^3=f$ then gives a covering of metric graphs as in Figure \ref{FIG:harmonic_ex}.

\end{exa}

Let $T$ be a metric tree with rational edge lengths. By redefining the valuation and subdividing the edges, we can find a finite model $T_{M}$ for $T$ with vertex set $V(T_{M})$, edge set $E(T_{M})$ and edges having length one. 
\begin{rem}
For the remainder of the section, we will only be using this particular model of $T$, so we will write $T=T_{M}$ for both the loopless model and the corresponding metric graph. 
\end{rem}

Let $R(E(T),\mathbb{F}_{p})$ be the ring of functions  
\begin{equation*}
J:E(T)\rightarrow{\mathbb{F}_{p}}.
\end{equation*}
Here addition and multiplication are induced from the ring structure on $\mathbb{F}_{p}$. One then easily checks that the units of this ring are given by functions $u:E(T)\rightarrow{\mathbb{F}_{p}^{*}}$.

We would like to associate a superelliptic covering of metric graphs to $T$ and $J(\cdot{})$. Imposing a condition on $J(\cdot{})$ which reflects the Riemann-Hurwitz equation allows us to do this. To that end, recall that for any superelliptic covering $X_{k}\rightarrow{\mathbb{P}^{1}_{k}}$ of degree $p$ with branch locus $B$, we have that
\begin{equation*}
g(X_{k})=\dfrac{2(1-p)+(p-1)|B|}{2}.
\end{equation*}
We will use this to define admissible coverings. For every vertex $v$, let $|B_{v}|$ be the number of edges adjacent to $v$ with $J(e)\neq{0\mod{p}}$. 
\begin{mydef}\label{HurwitzAdmissible}
Let $T$ be a model of a metric tree with rational edge lengths, let $E(T)$ be  its edge set and let $V(T)$ be its vertex set. A function $J:E(T)\rightarrow{\mathbb{F}_{p}}$ is called Riemann-Hurwitz admissible (or \emph{admissible}) if the following two conditions are satisfied: 
\begin{enumerate}
\item For every vertex $v\in{V(T)}$, the number
\begin{equation}\label{EquationGenus}
r_{v}:=\dfrac{2(1-p)+(p-1)|B_{v}|}{2}
\end{equation}
is an integer greater than or equal to zero. 
\item For every leaf edge $e$ in $T$, we have $J(e)\neq{0}\bmod{p}$. 
\end{enumerate}
\end{mydef}
The first condition allows us to construct a vertex $v'$ above $v$ with genus as predicted from the Riemann-Hurwitz formula. The second condition is added to ensure that the leaves have some nontrivial contribution to the vertex they are attached to. 
\begin{lemma}\label{Admissiblecoverings}
Let $J:E(T)\rightarrow{\mathbb{F}_{p}}$ be a Riemann-Hurwitz admissible function for a model $T$, 
as in Definition \ref{HurwitzAdmissible}. Then there exists a superelliptic covering $\psi:\Sigma\rightarrow{T}$ with the following properties:
\begin{enumerate}
\item There are $p$ edges lying above $e\in{T}$ if and only if $J(e)=0\mod{p}$.
\item There are $p$ vertices lying above a vertex $v\in{T}$ if and only if $|B_{v}|=0$. 
\item Let $v'\in{\Sigma}$. Then $g(v')=r_{v}$, as in Equation \ref{EquationGenus}.
\end{enumerate}

Conversely, we can construct an admissible function $J:E(T)\rightarrow{\mathbb{F}_{p}}$ for any superelliptic covering $\Sigma\rightarrow{T}$ as follows. We define
$$
J(e)={1}\mod{p}
$$
if and only if there is only edge lying above $e$ and 
$$
J(e)={0}\mod{p}
$$
if and only if there are $p$ edges lying above $e$.
\end{lemma}
\begin{proof}
We construct $\Sigma$ as follows. For every $v\in{T}$ with $|B_{v}|\neq{0}$, construct one vertex $v'$ with $g(v')=r_{v}$. For every edge $e\in{T}$ connecting to this $v$, construct $p$ edges with length one if $J(e)=0$ and $1$ edge with length $1/p$ if $J(e)\neq{0}$. 
For every vertex $v\in{T}$ with $|B_{v}|=0$, construct $p$ vertices $v_{i}$ with $g(v_{i})=0$ and for every adjacent edge $e$, construct $p$ edges $e_{i}$. Connecting these vertices and edges then yields the metric graph $\Sigma$, which has a natural $G:=\mathbb{Z}/p\mathbb{Z}$ action such that $\Sigma/G=T$. 
The function $J$ constructed in the lemma is admissible by the local Riemann-Hurwitz conditions.  
\end{proof}

\begin{rem}\label{RemarkJ}
We note that an admissible  $J(\cdot{})$ only fixes the superelliptic covering (in the sense of Lemma \ref{Admissiblecoverings}) up to a twisting of the edges, see Figure \ref{TwistingEdges}.
Furthermore, multiplying $J$ by any unit $u\in{(R(E(T),\mathbb{F}_{p}))^{*}}$ yields a function $u\cdot{J}$ that induces the same covering in the sense of Lemma~\ref{Admissiblecoverings}.  
\begin{figure}[ht]
\begin{center}
\includegraphics[width=2.5 in, height=1.0 in]{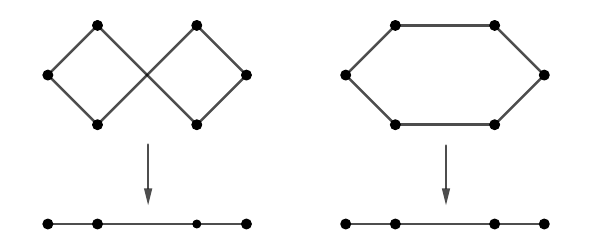}
\caption{Two $\mathbb{Z}/2\mathbb{Z}$-coverings whose associated admissible functions $J:E(T)\rightarrow{\mathbb{F}_{p}}$ from Lemma \ref{Admissiblecoverings} are the same. Here the edges are "twisted", which is to say that these coverings differ by an automorphism of order two. Note that the leaf edges have been surpressed in this example. } 
\label{TwistingEdges}
\end{center}
\end{figure}
\end{rem}
Let $\phi:\Sigma\rightarrow{T}$ be a superelliptic covering of metric graphs with rational edge lengths and let $J(\cdot{})$ be as in Lemma \ref{Admissiblecoverings}. 
Let $L(T)$ be the set of leaf edges in $T$ and let $R_{\phi}=|L(T)|$ be the number of leaf edges $e\in{T}$. 
This will also be referred to as the number of branch points of the covering. To show there exists an algebraic covering tropicalizing to the given covering, we construct $R_{\phi}$ points in $\mathbb{P}^{1}(K)$, labeled $P_{i}$,  and coefficients $a_{i}\in\mathbb{Z}$ such that the following are satisfied: 
\begin{enumerate}
\item The divisor $D:=\sum_{{i}}a_{i}(P_{i})$ has degree $0$.
\item Every point $P_{i}$ specializes to a unique leaf vertex $v(P_{i})\in{V(T)}$ under the specialization map $\rho():\mathrm{Div}(\mathcal{X})\rightarrow{\mathrm{Div}(\Sigma(\mathcal{X}))}$ for some strongly semistable model $\mathcal{X}$ for $\mathbb{P}^{1}$ with intersection graph $\Sigma(\mathcal{X})=T$.
\item There exists a unit $u\in{(R(E(T),\mathbb{F}_{p}))^{*}}$ such 
that the slope $\psi_{e}$ of the piecewise linear function $\psi$ corresponding to $\rho(D)=\Delta(\psi)$ satisfies
$$
\psi_{e}\bmod{p}=u(e)\cdot{}{J(e)}.
$$

\end{enumerate}
Since $D$ has degree zero, we can find an $f\in{K(x)}$ such that $\mathrm{div}(f)=D$. The third condition then ensures that the covering given by $(x,y)\mapsto{x}$ for $y^{p}=f$ tropicalizes to $\phi$. This can be seen using Proposition \ref{edgeprop}. The existence of the strongly semistable model $\mathcal{X}$ follows from the considerations in Section \ref{sepmodel}.

\begin{rem}
Throughout this section, we will consider piecewise linear functions $\psi$ on $T$ as functions from $V(T)$ to $\mathbb{Z}$. In the terminology of Section \ref{specialdivisor}, they are elements of $\mathcal{M}(T)$. This allows us to consider the base change of the space of piecewise linear functions to $\mathbb{F}_{p}$, see Equations \ref{BaseChange1} and \ref{BaseChange2}.  
\end{rem}

We first formalize the problem in terms of principal divisors on the tree $T$. We assume that we have already found $P_{i}\in\mathbb{P}^{1}(K)$ that uniquely specialize to the leaf vertices of $T$. The explicit construction of the $P_{i}$ and the strongly semistable model $\mathcal{X}$ will be given in Theorem \ref{realizabilitythm}. Let $M=\mathbb{Z}^{|L(T)|}$ and consider the $\mathbb{Z}$-submodule 
\begin{equation}\label{DegreeZeroHyperplane}
N=\left\{(x_{i})\in\mathbb{Z}^{|L(T)|}:\sum_{i=1}^{|L(T)|}x_{i}=0\right\}.
\end{equation}

There are then natural maps $j:N\rightarrow{\mathrm{Div}^{0}(\mathbb{P}^{1})}$ and $j_{\Sigma}:N\rightarrow{\mathrm{Div}^{0}(T)}$ given by
$$
j:(x_{i})\mapsto{\sum_{i=1}^{|L(T)|}}x_{i}\cdot{}(P_{i})
$$
and
$$
j_{T}:(x_{i})\mapsto{}{\sum_{i=1}^{|L(T)|}x_{i}\cdot{}\rho(P_{i})}.
$$
Together with the specialization map $\rho:\mathrm{Div}^{0}(\mathbb{P}^{1})\rightarrow{\mathrm{Div}^{0}(T)}$, these fit into the following commutative diagram:
\begin{equation*}
\begin{tikzcd}
N \arrow{r}{j} \arrow{dr}{j_{T}} & \mathrm{Div}^{0}(\mathbb{P}^{1}) \arrow{d}{\rho} \\
{} & \mathrm{Div}^{0}(T).
\end{tikzcd}
\end{equation*}
Since $T$ is a tree, we have $\mathrm{Div}^{0}(T)=\mathrm{Prin}(T)$ and we can thus write any degree zero divisor as $\Delta(\psi)$ for some piecewise linear function $\psi$.  
For any $a=(a_{i})\in{N}$, 
we write $\psi_{a}$ for a piecewise linear function on $T$ corresponding to the degree zero divisor $j_{T}((a_{i}))=\sum_{i}a_{i}\cdot{}\rho(P_{i})$. By base changing $j$ and $j_{T}$ over $\mathbb{F}_{p}$, we obtain the morphisms 
\begin{equation}\label{BaseChange1}
\overline{j}:N\otimes_{\mathbb{Z}}{\mathbb{F}_{p}}\rightarrow{\mathrm{Prin}(\mathbb{P}^{1})/p\cdot{}\mathrm{Prin(\mathbb{P}^{1})}}
\end{equation}
and
\begin{equation}\label{BaseChange2}
\overline{j}_{T}:N\otimes_{\mathbb{Z}}{\mathbb{F}_{p}}\rightarrow{\mathrm{Prin}(T)/p\cdot{}\mathrm{Prin(T)}}.
\end{equation}

We will now show that two degree zero divisors $j(a)$ and $j(a')$ in $\mathrm{Div}^{0}(\mathbb{P}^{1})$ that are equivalent under these specialization maps give rise to equivalent admissible functions $J(\cdot{})$. 
Here the admissible function $J(\cdot{})$ for a degree zero divisor $j(a)$ is obtained as follows. 

We write $j((a_{i}))=\mathrm{div}(f)$ for some $f$ and then consider the algebraic covering $C_{a}\rightarrow{\mathbb{P}^{1}}$ induced by the extension of function fields
\begin{equation*}
K(x)\rightarrow{K(x)[z]/(z^{p}-f)}.
\end{equation*}
This algebraic covering then gives rise to a superelliptic covering of metric graphs $\Sigma_{a}\rightarrow{T}$ by Proposition \ref{disbranprop1}.

By Lemma \ref{Admissiblecoverings}, we obtain the desired admissible function $J(\cdot{})$. 

\begin{lemma}
\label{EquivalentCoverings1}
Let $a,a'\in{N}$ be such that 
\begin{equation*}
\overline{j}_{T}(\overline{a})=\overline{j}_{T}(\overline{a}'),
\end{equation*}
where $\overline{a}$ (resp. $\overline{a}'$) denotes the image of $a$ (resp. ${a}'$) in $N\otimes_{\mathbb{Z}}{\mathbb{F}_{p}}$.
Then the superelliptic coverings of metric graphs 
$\phi_{a}:\Sigma_{a}\rightarrow{T}$ and $\phi_{a'}:\Sigma_{a'}\rightarrow{T}$ corresponding to the degree zero divisors $j(a)$ and $j(a')$ are the same in the sense of Lemma \ref{Admissiblecoverings}. 
\end{lemma}
\begin{proof}
The values of the piecewise linear functions $\psi_{a}$ and $\psi_{a'}$ corresponding to $j_{T}(a)$ and $j_{T}(a')$ differ at every vertex by a multiple of $p$. The slopes on every edge then also differ by a multiple of $p$. Using Proposition \ref{edgeprop}, we then see that the number of vertices and edges lying above every vertex and edge of $T$ is the same for both $\phi_{a}:\Sigma_{a}\rightarrow{T}$ and $\phi_{a'}:\Sigma_{a'}\rightarrow{T}$. We thus see that the coverings yield the same admissible function $J:E(T)\rightarrow{\mathbb{F}_{p}}$, as desired. 

For another proof, note that any divisor of the form $pD\in\mathrm{Div}^{0}(T)$ on the tree corresponds to the divisor of a function $g(x)^{p}$ for some $g(x)\in{K(x)}$. If we then have $j(a)=pD+j(a')$, then the algebraic curves $z^{p}=f_{a}(x)$ and $z'^{p}=f_{a}(x)\cdot{g(x)^{p}}$ are isomorphic by the transformation 
\begin{equation*}
(x,z)\mapsto{(x,z\cdot{}{g})}.
\end{equation*}
The corresponding metric graphs are then also isomorphic, as desired.  
\end{proof}

For the remainder of the section, we fix a target vertex $v_{0}$ with at least two branch points, say $P_{0,1}$ and $P_{0,2}$. We want to construct a divisor $j((a_{i}))=\sum_{i=1}^{|L(E)|}a_{i}(P_{i})$ that gives rise to the covering $\phi$. We will first assign values to the $a_{i}$ without assuming that the vector $(a_{i})$ belongs to the degree zero submodule $N$, as introduced in Equation \ref{DegreeZeroHyperplane}. Let $a_{0,1}$ and $a_{0,2}$ be the coefficients corresponding to $P_{0,1}$ and $P_{0,2}$. 
In the argument that follows, we will not assign any values to these coefficients until the very end. These coefficients $a_{0,1}$ and $a_{0,2}$ are then used to ensure that $(a_{i})\in{N}$.
For any edge $e$ in $T$, consider the connected component $T_{e}$ of $T\backslash\{e\}$ not containing $v_{0}$ as in Lemma \ref{slopelemma}. The slope of a rational function giving this divisor along an edge $e$ is now given by the formula in Lemma~\ref{slopelemma}: 
\begin{equation*}
\sum_{x \in T_e} (\psi)(x)=\sum_{P_{i}\in{T_{e}}}{a_{i}}.
\end{equation*}

\begin{mydef}
Let $s_{e}$ be the number of $P_{i}$ reducing to the connected component $T_{e}$. 
The \emph{total Laplacian} on the component $T_{e}$ is a function 
$
\Delta_{e}(\psi):(\mathbb{F}_{p})^{s_{e}}\rightarrow{\mathbb{F}_{p}},
$
sending 
$
(a_{i})\mapsto \sum_{P_{i}\in{T_{e}}}{a_{i}}.
$ We consider these as elements of $\mathbb{F}_p$ because we are only interested in the value of the slopes and the exponents mod $p$, see Lemma \ref{EquivalentCoverings1}. 
\end{mydef}

This definition allows us to view the formula for the slope of the Laplacian on $e$ as a function of the $a_{i}$ laying on the connected component $T_{e}$. 
The Laplacian $\psi$ corresponding to the covering $\phi$ must satisfy the {\emph{total Laplacian equations}} for every $e\in{E(\Sigma)}$:

\begin{equation*}
\Delta_{e}(\psi)(a(e))=u(e)\cdot{J(e)},
\end{equation*}
where $u(e)\in\mathbb{F}_{p}^{*}$ and $a(e)$ is the vector of coefficients $(a_{i})$ corresponding to points $P_{i}$ reducing to leaves of $T_{e}$.

We adopt the following notation for the rest of this section. Here, we use the symbols $\overline{0}$ and $\overline{1}$ for the residue classes of $0$ and $1$ in $\mathbb{F}_{p}$.

\begin{enumerate}
\item We write $\Delta_{e}(\psi)={{1}}$ if there exist $a_{i}$ such that $\Delta_{e}(\psi)(a(e))\in{\mathbb{F}^{*}_{p}}$. 
\item Similarly, we write $\Delta_{e}(\psi)={{0}}$ if there exist $a_{i}$ such that $\Delta_{e}(\psi)(a(e))=\overline{0}$.
\item Given a set of edges $E:=\{e_{i}\}$, we write $\Delta_{E}(\psi)=\delta_{e_{i}}$ for $\delta_{e_{i}}\in\{{0},{1}\}$ if there exist $a_{i}$ such that all conditions $\Delta_{e_{i}}(\psi)=\delta_{e_{i}}$ are met simultaneously for this set of $\{a_{i}\}$. 
\item Given an edge $e$ with connected component $T_{e}$ and numbers $\delta_{e_{i}}\in\{{0},{1}\}$ for $e_{i}\in{T_{e}}$, we write $\Delta_{T_{e}}(\psi)=c$ for a $c\in\mathbb{F}_{p}$ if there exist $a_{i}$ such that $\Delta_{e_{j}}(\psi)={\delta_{e_{j}}}$ for every $e_{j}$ in $T_{e}$ and such that $\Delta_{e}(\psi)(a_{e})=c$.
\end{enumerate}
With this notation, the covering $\phi:\Sigma\rightarrow{T}$ naturally gives us a set of $\{\delta_{e_{i}}\}$ with $\delta_{e_{i}}\in\{{0},{1}\}$: we set $\delta_{e_{i}}=0$ if $J(e_{i})=\overline{0}$ and $\delta_{e_{i}}=1$ if $J(e_{i})=\overline{1}$, where $J(\cdot{})$ is as given by Lemma \ref{Admissiblecoverings}.

\begin{exa}
Let $P_{i}\in{\mathbb{P}^{1}(K)}$ for $i\in\{5,6,7,8,9,10\}$ with tropical tree $T$ be as in Example \ref{RealizabilityExample1}. The fixed vertex for this example will be the image of $v\in{\Sigma}$ in $T$. We will be more general here and take a prime number $p$ not equal to $2$. The submodule $N$ as introduced in Equation \ref{DegreeZeroHyperplane} is given here by 
\begin{equation*}
N=\{(c_{5},c_{6},c_{7},c_{8},c_{9},c_{10}):c_{5}+c_{6}+c_{7}+c_{8}+c_{9}+c_{10}=0\}.
\end{equation*}

The total Laplacian $\Delta_{f_{1}}(\psi):(\mathbb{F}_{p})^{2}\rightarrow{\mathbb{F}_{p}}$ for $f_{1}\in{T}$ is then given by  $$(\overline{c}_{5},\overline{c}_{6})\mapsto{\overline{c}_{5}+\overline{c}_{6}}.$$
By choosing suitable $c_{5}$ and $c_{6}$, we see that $\Delta_{f_{1}}(\psi)={0}$ and $\Delta_{f_{1}}(\psi)={1}$. Indeed, take $(\overline{c}_{5},\overline{c}_{6})=(\overline{1},\overline{p-1})$ and $(\overline{c}_{5},\overline{c}_{6})=(\overline{1},\overline{p-2})$ respectively. Note that the assumption $p\neq{2}$ is crucial here. If $p=2$, then we cannot obtain $\Delta_{f_{1}}(\psi)={1}$, as the slope is automatically equal to $0\bmod{2}$ by our assumption $c_{5},c_{6}\equiv{1}\bmod{2}$ (see Definition \ref{HurwitzAdmissible}). 

For $f_{4}$, the total Laplacian $\Delta_{f_{4}}(\psi):(\mathbb{F}_{p})^{3}\rightarrow{\mathbb{F}_{p}}$ is given by
\begin{equation*}
(\overline{c}_{5},\overline{c}_{6},\overline{c}_{7})\mapsto{\overline{c}_{5}+\overline{c}_{6}+\overline{c}_{7}}.
\end{equation*}
Suppose we are given $\delta_{f_{1}}=0$ or $\delta_{f_{1}}=1$. 
We now claim that $\Delta_{T_{f_{4}}}(\psi)=c$ for any $c\in{\mathbb{F}^{*}_{p}}$. 
Indeed, suppose that $\delta_{f_{1}}=0$. Then $\overline{c}_{5}+\overline{c}_{6}=\overline{0}$ and we take $\overline{c}_{7}=c\in(\mathbb{F}_{p})^{*}$. Suppose that $\delta_{f_{1}}=1$. If $c\neq{\overline{2}}$, then we take the vector $$(\overline{c}_{5},\overline{c}_{6},\overline{c}_{7})=(\overline{1},\overline{1},c-\overline{2}).$$
If $c=\overline{2}$, then we take the vector $$(\overline{c}_{5},\overline{c}_{6},\overline{c}_{7})=(\overline{1},\overline{2},\overline{-1})$$. This shows that $\Delta_{T_{f_{4}}}(\psi)=c$ for any $c\in{\mathbb{F}^{*}_{p}}$ and any choice of $\delta_{f_{1}}\in\{0,1\}$.  
\end{exa}

\begin{lemma} Let $\Sigma \rightarrow T$ be a superelliptic covering of metric graphs with rational edge lengths. Let $e \in T$ be an edge. 
\begin{enumerate}
\item If $e$ is ramified, then $\Delta_{T_{e}}(\psi)=c$ for any $c\in\mathbb{F}^{*}_{p}$. 
\item If $e$ is unramified, then $\Delta_{T_{e}}(\psi)=\overline{0}$.
\end{enumerate}
\label{keylemma}
\end{lemma}
\begin{proof}
We prove the lemma by induction on $|E(T_{e})|$. The inductive hypothesis is
\begin{center}
$I_{n}$: For every $e$ such that $|E(T_{e})|\leq{n}$, we have $\Delta_{T_{e}}(\psi)=0$ if $e$ is unramified and $\Delta_{T_{e}}(\psi)=c$ for every $c\in\mathbb{F}^{*}_{p}$ if $e$ is ramified. 
\end{center}
For $n=0$, $T_{e}$ consists of a single vertex $v$. In other words, $e$ is a leaf edge with only one point $P_{e}$ associated to it. For $D=\sum_{e\in{L(T)}}c_{e}(P_{e})$, we then have that $\delta_{e}(\psi)=c_{e}$, which we can choose freely. Thus $\Delta_{T_{e}}(\psi)=c$ for any $c\in\mathbb{F}_{p}^{*}$. 

Now suppose the statement is true for $n$. Let $e$ be any edge such that $|E(T_{e})|=n+1$. Let $v$ be the vertex in $T_{e}$ connected to $e$. Then for every other edge connecting to $v$, we have $|E(T_{e_{i}})|\leq{n}$, so by the induction hypothesis we know the statement is true for $T_{e_{i}}$. Suppose $e$ is ramified (that is, $J(e)=1\bmod{p}$). Then $v$ is connected to another edge $e'$ that is ramified (which can be a leaf edge, coming from a branch point $P_{e}$ or another edge with $J(e)=1\bmod{p}$). Indeed, otherwise the covering at $v$ would only be ramified at one point, which is impossible by the Riemann-Hurwitz conditions. If it is the only ramified edge, we are done because $\Delta_{T_{e}}(\psi)=\Delta_{T_{e'}}(\psi)$. Otherwise, there are at least two ramified edges connected to $v$. For $p\neq{2}$, we can adjust these two values freely to obtain any value $c\in\mathbb{F}_{p}^{*}$. In other words, $\Delta_{T_{e}}(\psi)=c$ for any $c\in\mathbb{F}_{p}^{*}$. If $p=2$, then there must be at least three other ramified edges connected to $v$, because otherwise the covering would be unramified at $e$. We then again easily obtain $\Delta_{T_{e}}(\psi)=\overline{1}\in(\mathbb{F}_{2})^{*}$, as desired.  

Suppose that $e$ is unramified and let $v$ again be the vertex in $T_{e}$ that is connected to $e$. The proof in this case is similar. If every other edge connecting to $v$ is unramified, then $\Delta_{T_{e}}(\psi)=0$. Otherwise, there are at least two edges that are ramified. Using the induction hypothesis on these edges and Lemma \ref{slopelemma}, we obtain $\Delta_{T_{e}}(\psi)=0$, as desired. 
\end{proof}

We apply Lemma \ref{keylemma} for the only nonleaf edge connected to the vertex $v_{0}$ to obtain an assignment for all $a_i$ except the ones corresponding to the two remaining leaf edges. We then use these last coefficients to make the degree of the divisor zero. This gives us the following corollary. 

\begin{cor}\label{maincor}
Given any superelliptic covering $\Sigma\rightarrow{T}$ with covering data $\delta_{e}$ for every edge, we have  $\Delta_{E}(\psi)=\delta_{e}$ for $E=E(T)$. The divisor $D=\sum_{i}a_{i}(P_{i})$ corresponding to this solution can be chosen to satisfy $\mathrm{deg}(D)=0$. 
\end{cor}

\begin{theorem}
\label{realizabilitythm}
Let $p$ be a prime number. A covering $\phi_\Sigma:\Sigma \rightarrow T$ is a superelliptic covering of degree $p$ of weighted metric graphs with rational edge lengths if and only if there exists a superelliptic covering $\phi:X\rightarrow \mathbb{P}^1$ of degree $p$ tropicalizing to it. 
\end{theorem}
\begin{proof}
Suppose we have a superelliptic admissible covering of graphs $\phi_\Sigma:\Sigma \rightarrow T$ of degree $p$. We present a procedure for constructing a polynomial $f$ such that the covering from the curve $y^p = f(x)$ defined by $(x,y) \mapsto x$ tropicalizes to $\phi_\Sigma:\Sigma \rightarrow T$.
\begin{enumerate}
\item On each vertex $v_i \in T$, use the local Riemann-Hurwitz condition to determine the number of leaf edges $r(v_i)$ needed on each vertex.
\item Each vertex $v_i \in T$ corresponds to a collection of points $P_{i,1}, \ldots, P_{i,r(v_i)} \in \mathbb{P}^1(K)$, each corresponding to the leaf edges attached at $v_i$. We assume that these points are contained in the affine $\mathbb{P}^{1}\backslash\{\infty\}$. The equation for $f$ is then
$
f(x) = \prod_{v_i \in T}\prod_{j=1}^{r(v_i)}(x-P_{i,j})^{a_{ij}}.
$

\item Find the $a_{ij}$ as follows.
Select a target vertex $v_{0}$ with at least two leaf edges. For every edge in the graph, we solve the corresponding total Laplacian equation with respect to $v_{0}$. The fact that there is a solution follows from Corollary \ref{maincor}. 
Pick a solution to these equations. Consider the branch points $P_{v_{0},1},...,P_{v_{0},s}$ reducing to $v_{0}$. The valuations at these points satisfy
\begin{equation*}
\sum_{i=1}^{s}a_{v_{0},i}=\sum_{P \text{ not reducing to }v_{0}}-a_{P},
\end{equation*}
Picking values for the $a_{v_{0},i}$ that satisfy this equation concludes the algorithm for finding the $a_{ij}$. 

\item To obtain the desired points $P_{i}$, we view these trees as describing $\pi$-adic expansions of elements in $K$. To be explicit, let $S$ be a set of representatives for the residue field $k$. Let $v_{0}$ be an endpoint of $T$, and let $v_{1}$ be the vertex connected to $v_{0}$. For every leaf edge $e$ (with end vertex not equal to $v_{0}$) attached to $v_{1}$, construct a point $P_{e}=c_{e}\pi$, with the $c_{e}\in{S}$ distinct. This might require a finite extension of the residue field $k$, which corresponds to a finite (unramified) extension of $K$. 
For every nonleaf edge $e_{i}$, take an element $c_{i}\in{S}$ that is not equal to the $c_{e}$. For such an edge $e_{i}$, consider the connecting vertex $v_{1,i}$. For every leaf edge ${e}$ attached to $v_{1,i}$, find distinct $c_{i,e}\in{S}$ and construct $P_{e}=c_{i}\pi+c_{i,e}\pi^{2}$. For every nonleaf edge $e_{i,j}$ connected to $v_{1,i}$, repeat the procedure and construct elements $c_{i,j}\in{S}$ distinct from the $c_{i,e}$, where $e$ is a leaf edge. We do one more step of the inductive procedure. Let $v_{1,i,j}$ be the other vertex connected to $e_{i,j}$. For every leaf edge $e$ attached to $v_{1,i,j}$, find distinct $c_{i,j,e}$ and construct $P_{e}=c_{i}\pi+c_{i,j}\pi^{2}+c_{i,j,e}\pi^{3}$. At some point, we reach vertices that only have leaf edges as neighboring edges. At this point, we stop the procedure and find a set of points $\{P_{e}\}$. The tree corresponding to this set of points is $T$. On the algebraic side, we can take the canonical semistable $\mathcal{Y}$ corresponding to this set (see Section \ref{sepmodel}) and its intersection graph $\Sigma_{\mathcal{Y}}$ is $T$ minus the leaf edges. 
\end{enumerate}

This concludes the realizability part of the theorem. The backwards direction is obtained by combining \cite[Theorem 4.6.1.]{tropabelian} and \cite[Chapter 10, Proposition 3.48, Page 526]{liu2} for the edge lengths. 
\end{proof}

A natural question following from this is whether the same result holds for non-prime integers $n$. We conjecture that this is indeed the case and that a similar proof could be used.

\section{Moduli Spaces}
\label{modulispaces}

The moduli space $M_g^{\text{tr}}$ of weighted metric graphs of genus $g$ was defined in \cite{BMV}, and{} has the structure of a $(3g-3)$-dimensional stacky fan. The cones in $M_g^{\text{tr}}$ of dimension $d$ correspond to \emph{combinatorial types}, which are pairs consisting of a graph $H$ with $d$ edges and a weight function $w$ on its vertices.
A \emph{constrained type} is a triple $(H,w,r)$, where $r$ is an equivalence relation on the edges of $H$. In a metric graph $\Sigma$ corresponding to the constrained type $(H,w,r)$, the equivalence relation $r$ requires that edges in the same equivalence class have the same length. One can contract edges of a constrained type to arrive at a new constrained type. The operation of contraction is discussed in detail in \cite[Section 4.1]{trophyp} and depicted in Figure \ref{d3g4}. 
\begin{mydef}
The \emph{moduli space of tropical superelliptic curves} $S_{g,n}^{\text{tr}}$ is the set of weighted metric graphs of genus $g$ which have a degree $n$ superelliptic covering to a tree. Let $Sp_{g,n}^{\text{tr}}\subset S_{g,n}^{\text{tr}}$ denote the image under tropicalization of superelliptic curves defined by equations of the form $y^n=f(x)$ with distinct roots. 

The set $S_{g,n,rat}^{\text{tr}}$ is the subset of $S_{g,n}^{\text{tr}}$ consisting of those weighted metric graphs that have rational edge lengths. We similarly define $Sp_{g,n,rat}^{\text{tr}}$.
\end{mydef}

By Theorem \ref{realizabilitythm}, when $n$ is prime we have $S_{g,n,rat}^{\text{tr}}\subset M_g^{\text{tr}}$ is equal to the image under tropicalization of the locus of superelliptic curves (defined over a finite extension of $K$) inside $M_g$, the moduli space of genus $g$ curves. We comment $Sp_{g,n}^{\text{tr}} \subsetneq S_{g,n}^{\text{tr}}$ when $n >2$. See Figure \ref{d3g4} for the combinatorial types of weighted metric graphs corresponding to cones inside $S_{4,3}^{\text{tr}}$ and $Sp_{4,3}^{\text{tr}}$.

\begin{figure}[ht]
\begin{center}
\includegraphics[width=0.9\textwidth, height=3.6 in]{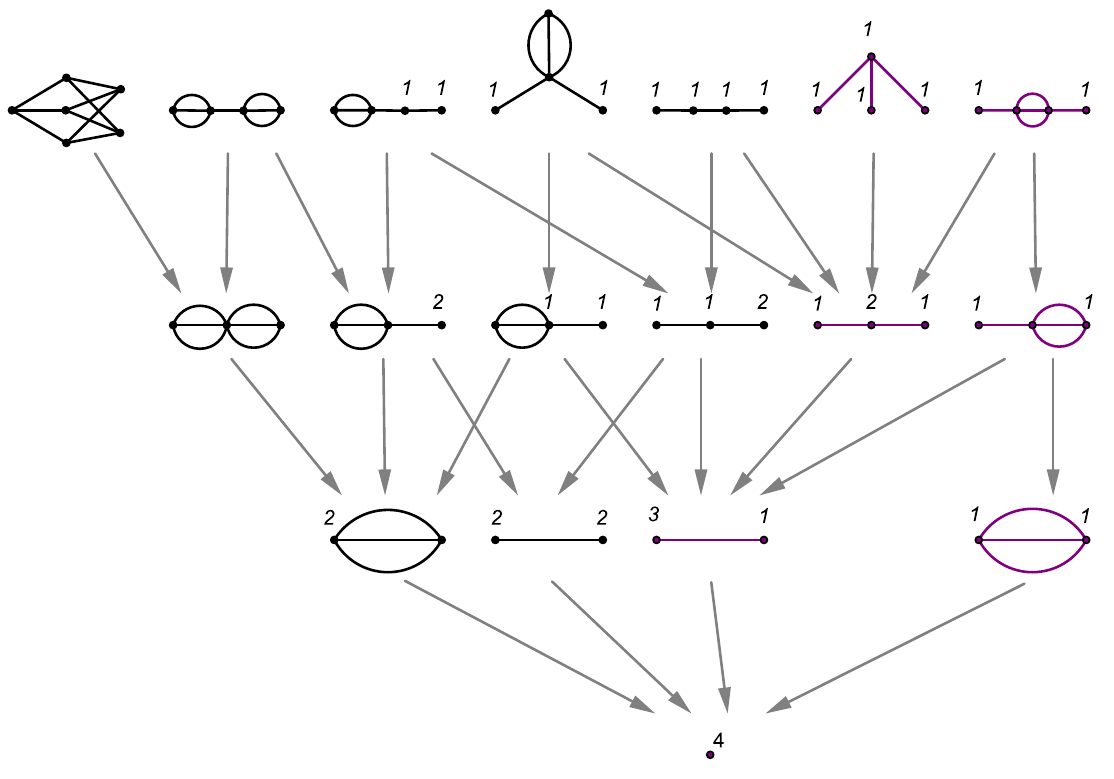}
\caption{The combinatorial types of weighted metric graphs corresponding to maximal cones in $S_{4,3}^{\text{tr}}$. The graphs which are also inside $Sp_{4,3}^{\text{tr}}$ are purple. The arrows between them correspond to contractions of edges.}
\label{d3g4}
\end{center}
\end{figure}

\begin{prop} The locus $S_{g,n}^{\text{tr}}$ of weighted metric graphs of genus $g$ which have a degree $n$ superelliptic covering to a tree has the structure of a 
stacky polyhedral fan.
\end{prop}

\begin{proof}
In \cite[Section 4.1]{trophyp}, Chan proves that given a collection $S$ of constrained types which are closed under contraction, the space $M_S$ they define is a stacky fan with cells in correspondence with the constrained types. 
We can obtain a constrained type from a combinatorial type of genus $g$ graph with a degree $n$ superelliptic covering to a tree by making the relation $r$ to equate any edges which have the same image under the covering map. Let $(H, w,r)$ be such a type, admitting a degree $n$ superelliptic covering $\theta$ to the tree $T$ which gives the relation $r$. 
 If $(H',w',r')$ is a contraction of $(H, w,r)$ along the equivalence class of edges $[e]$, and $T'$ is the contraction of $T$ along the edge $\theta([e])$, one can see using the local Riemann-Hurwitz equations $(H',w',r')$ admits a degree $n$ superelliptic covering to the tree $T'$. 
\end{proof}

Using the Riemann-Hurwitz equation, we can compute the genus of a graph in the case when $p$ is prime and the map has $r$ ramification points. To that end, let
$
g(p,r) := (p-1)(r/2-1).
$

\begin{theorem}
\label{hi}
Let $r \geq 4$ be an integer number of ramification points. Given two odd primes $p$ and $p'$, the stacky polyhedral fan $S_{g(p,r),p}^{\text{trop}}$ is the same as $S_{g(p',r),p'}^{\text{trop}}$.
\end{theorem}
\begin{proof}
Define a new stacky polyhedral fan $Ts(r)$ for $r \geq 4$ whose cones correspond to pairs $(T,s)$, where $T$ is a tree on $r$ leaves and $s$ is any subset of the edges of $T$ which we call a \emph{signature}. Each cone has dimension equal to the number of interior edges of $T$. We glue the cone $(T,s)$ to the cone $(T',s')$ when $(T',s')$ is a pair in which an edge $e \in T$ has been contracted, and $s'= s \backslash \{e\}$.

Given a tree $T$ with $r$ leaves and $m$ interior edges, one can compute all superelliptic graphs with a degree $p$ map to $T$ in the following way. A choice of signature on $T$ corresponds to deciding which interior edges have $p$ preimages or $1$ preimage in a superelliptic graph mapping to $T$. This yields $2^{m}$ signatures, but some signatures do not yield admissible covers. On each interior vertex $v$, compute the weight using the local Riemann-Hurwitz equation. If the vertex has no leaves and all edges adjacent to it have multiple preimages, then the vertex has $p$ preimages and weight 0. Otherwise
$
w(v) = (p-1)(r_v-2)/2,
$
where $r_v$ is the number of leaves at $v$ plus the number of ramified edges.
The graph is superelliptic if and only if this number is a positive integer for all vertices of the tree. Since $p$ is odd, this is always an integer. Then, we need that at each vertex, $r_v \geq 2$. Any signature on a tree satisfying this yields a superelliptic graph. In other words, the graphs admitting a degree $p$ superelliptic cover of $T$ are in bijection with the good choices of signatures on $T$. The space $S_{g(p,r),p}^{\text{tr}}$ naturally sits inside $Ts(r)$; each cone corresponding to a superelliptic curve $\Sigma \rightarrow T$ is mapped to the cone $(T,s)$ where $T$ is the tree corresponding to that curve and $s$ is the set of ramified edges in the covering. Whether or not a signature is admissible does not depend on $p$, so for any odd primes $p$ and $p'$, the images $S_{g(p,r),p}^{\text{tr}} \subset Ts(r)$ and $S_{g(p',r),p'}^{\text{tr}} \subset Ts(r)$ are the same.
\end{proof}

\begin{theorem} For primes $p \leq 17$ and number of ramification points $r\leq 14$, the number of maximal cones in $S_{g(p,r),p}^{\text{tr}}$ and $Sp_{g(p,r),p}^{\text{tr}}$ is given in Table~\ref{computation}.
\end{theorem}

\begin{proof}
This was done by direct computation in \texttt{Mathematica}, as we describe below. We restrict to the case of prime $n=p$ to simplify the computation. First, we precompute all trivalent trees on $r$ leaves. 

For the computations for $S_{g(p,r),p}^{\text{tr}}$, we create all graphs arising from assigning on each interior edge of the tree which ones are ramified and which are unramified (this gives $2^{r-3}$ possibilities). We check which of the resulting graphs have an assignment of non-negative integer weights on the vertices satisfying the local Riemann-Hurwitz condition. Then, we remove the isomorphic duplicates. For the case of fourteen leaves, this took 2.6 days to compute.

For the computations on $Sp_{g(p,r),p}^{\text{tr}}$, the possible metric graphs arising from a fixed tree $T$ depend only on the choice of where $\infty$ specializes. This is because the metric graph is determined by the divisor of $f$, and in the case with distinct roots, this is completely determined by where $\infty$ specializes on $T$. So, we make all possible choices and compute the resulting metric graphs. Then, we remove the isomorphic duplicates. The computations with twenty leaves took 16 hours each to compute.
\end{proof}

\begin{table}[ht]
\begin{center}
\begin{tabular}{c | c | c | c c c c c c c c }
 	& & $S_{g,p}^{\text{tr}}$ & & & & $Sp_{g,p}^{\text{tr}}$ & & & \\
\hline
r	&p=2 	& p>2			& 3 		& 5	 	& 7 		& 11 		&13		&17\\
	\hline
4 	&1		&2			&1		&1		&1		&1		&1		&1\\
5 	&0		&2				&2		&1		&1		&1		&1		&1\\
6 	&2		&7				&2		&0		&2		&2		&2		&2\\
7 	& 0		&11				&0		&5		&2		&2		&2		&2\\
8 	& 4		&34				&11		&7		&0		&4		&4		&4\\
9 	& 0		&80				&6		&12		&17		&6		&6		&6\\
10 	& 11	&242			&0		&11		&22		&11		&11		&11\\
11 	& 0		&682			&92		&0		&40		&18		&18		&18\\
12 	& 37	&2146			&37		&160	&70		&0		&37		&37\\
13 	& 0		&6624			&0		&227	&132	&273	&66		&66\\
14	&135	&21447			&916	&457	&135	&342	&0		&135\\
15	&0		&-				&265	&265	&0		&679	&1248	&265\\
16	&552	&-				&0		&0		&3167	&1173	&1535	&552\\
17	&0		&-				&10069	&8011	&4323	&2374	&3098	&1132\\
18 	& 2410	&-				&2410	&12029	&8913	&4687	&5359 	&0\\
19	&0		&-				&0		&24979	&16398	&9859	&10996	&29729\\
20	&11020	&-				&117746	&11020	&34511	&20542	&21833	&35651\\
\end{tabular}
\caption{The number of maximal cones in $Sp_{g(p,r),p}^{\text{tr}}$ and $S_{g(p,r),p}^{\text{tr}}$. The column labeled $p=2$ displays the number of maximal cones in $Sp_{g(2,r),2}^\text{tr}$ and $S_{g(2,r),2}^\text{tr}$.
}
\label{computation}
\end{center}
\end{table}

\subsubsection*{Acknowledgements.} The authors would like to thank the Max Planck Institute for Mathematics in the Sciences for their hospitality while they carried out this project. Madeline Brandt was supported by a National Science Foundation Graduate Research Fellowship. 


{\bibliographystyle{plain}
\bibliography{sample}}

\begin{thebibliography}{10}

\bibitem{ABBR1}
Omid Amini, Matthew Baker, Erwan Brugall{\'{e}}, and Joseph Rabinoff.
\newblock Lifting harmonic morphisms {I}: {M}etrized complexes and {B}erkovich
  skeleta.
\newblock {\em Springer, Research in the Mathematical Sciences}, 2(1), June
  2015.

\bibitem{Atiyah}
Michael Atiyah and Ian Macdonald.
\newblock {\em Introduction To Commutative Algebra (Addison-Wesley Series in
  Mathematics)}.
\newblock CRC Press, 1994.

\bibitem{baker}
Matthew Baker.
\newblock Specialization of linear systems from curves to graphs.
\newblock In {\em Algebra and Number Theory}, volume~2, pages 613--653. 01
  2008.

\bibitem{bakerfaber}
Matthew Baker and Xander Faber.
\newblock Metrized graphs, electrical networks, and {F}ourier analysis.
\newblock {\em arXiv:math/0407428}, 2004.

\bibitem{berkovich2012}
V.G. Berkovich.
\newblock {\em Spectral Theory and Analytic Geometry over Non-Archimedean
  Fields}.
\newblock Mathematical Surveys and Monographs. American Mathematical Society,
  2012.

\bibitem{bbc}
Barbara Bolognese, Madeline Brandt, and Lynn Chua.
\newblock From curves to tropical jacobians and back.
\newblock In G.G. Smith and B.~Sturmfels, editors, {\em Combinatorial Algebraic
  Geometry}. to appear.

\bibitem{bosch2014}
Siegfried Bosch.
\newblock {\em Lectures on Formal and Rigid Geometry}.
\newblock Springer International Publishing, 2014.

\bibitem{magma}
Wieb Bosma, John Cannon, and Catherine Playoust.
\newblock The {M}agma algebra system. {I}. {T}he user language.
\newblock {\em J. Symbolic Comput.}, 24(3-4):235--265, 1997.
\newblock Computational algebra and number theory (London, 1993).

\bibitem{BMV}
Silvia Brannetti, Margarida Melo, and Filippo Viviani.
\newblock On the tropical {T}orelli map.
\newblock {\em Adv. Math.}, 226(3):2546--2586, 2011.

\bibitem{admcov}
Renzo Cavalieri, Hannah Markwig, and Dhruv Ranganathan.
\newblock Tropicalizing the space of admissible covers.
\newblock {\em Mathematische Annalen}, 364(3-4):1275--1313, 2016.

\bibitem{trophyp}
Melody Chan.
\newblock Tropical hyperelliptic curves.
\newblock {\em Journal of Algebraic Combinatorics. An International Journal},
  37(2):331--359, 2013.

\bibitem{SGA1}
Alexander Grothendieck.
\newblock {\em Rev\^etements \'Etales et Groupe Fondamental (SGA 1)}, volume
  224 of {\em Lecture notes in mathematics}.
\newblock Springer-Verlag, 1971.

\bibitem{shimura}
Samuel Grushevsky and Martin Moeller.
\newblock Explicit formulas for infinitely many {S}himura curves in genus 4.
\newblock {\em Asian J. Math}, to appear.

\bibitem{igusa}
Paul~Alexander Helminck.
\newblock Tropical {I}gusa invariants and torsion embeddings.
\newblock {\em arXiv:1604.03987}, 2016.

\bibitem{tropabelian}
Paul~Alexander Helminck.
\newblock Tropicalizing abelian covers of algebraic curves.
\newblock {\em arXiv:1703.03067v2}, 2018.

\bibitem{karpilovsky1989topics}
G.~Karpilovsky.
\newblock {\em Topics in {F}ield {T}heory}.
\newblock ISSN. Elsevier Science, 1989.

\bibitem{liu}
Qing Liu.
\newblock Courbes stables de genre 2 et leur schéma de modules.
\newblock {\em Mathematische Annalen}, 295(2):201--222, 1993.

\bibitem{liu2}
Qing Liu.
\newblock {\em Algebraic Geometry and Arithmetic Curves}.
\newblock Oxford Graduate Texts in Mathematics (Book 6). Oxford University
  Press, 2006.

\bibitem{liu_lorenzini_1999}
Qing Liu and Dino Lorenzini.
\newblock Models of curves and finite covers.
\newblock {\em Compositio Mathematica}, 118(1):61–102, 1999.

\bibitem{tropicalbook}
D.~Maclagan and B.~Sturmfels.
\newblock {\em Introduction to Tropical Geometry:}.
\newblock Graduate Studies in Mathematics. American Mathematical Society, 2015.

\bibitem{shimura2}
Martin {M}{\"{o}}ller.
\newblock Shimura and {T}eichm\"uller curves.
\newblock {\em J. Mod. Dyn.}, 5(1):1--32, 2011.

\bibitem{Maple}
Michael~B. Monagan, Keith~O. Geddes, K.~Michael Heal, George Labahn, Stefan~M.
  Vorkoetter, James McCarron, and Paul DeMarco.
\newblock {\em Maple~10 Programming Guide}.
\newblock Maplesoft, Waterloo ON, Canada, 2005.

\bibitem{neu}
J\"{u}rgen Neukirch.
\newblock {\em Algebraic Number Theory}.
\newblock Grundlehren der mathematischen Wissenschaften. Springer Berlin
  Heidelberg, 1999.

\bibitem{PS05}
Lior Pachter and Bernd Sturmfels.
\newblock {\em Algebraic Statistics for Computational Biology}.
\newblock Cambridge University Press, New York, NY, USA, 2005.

\bibitem{RSS}
Qingchun Ren, Steven~V. Sam, and Bernd Sturmfels.
\newblock Tropicalization of classical moduli spaces.
\newblock {\em Mathematics in Computer Science}, 8(2):119--145, 2014.

\bibitem{silv2}
Joseph~H. Silverman.
\newblock {\em Advanced Topics in the Arithmetic of Elliptic Curves}.
\newblock Graduate Texts in Mathematics. Springer-Verlag New York, 1994.

\bibitem{stacks-project}
The {Stacks Project Authors}.
\newblock exit{Stacks Project}.
\newblock \url{http://stacks.math.columbia.edu}, 2017.

\bibitem{Ste2}
Peter Stevenhagen.
\newblock Voortgezette getaltheorie.
\newblock \url{http://websites.math.leidenuniv.nl/algebra/localfields.pdf},
  2002.

\bibitem{Sticht2009}
Henning Stichtenoth.
\newblock {\em Algebraic Function Fields and Codes}.
\newblock Springer Berlin Heidelberg, 2009.

\bibitem{Temkin2018}
Michael Temkin.
\newblock Moduli spaces of $n$-pointed stable curves.
\newblock
  \url{http://www.math.huji.ac.il/~temkin/teach/math647/lecture_notes.pdf},
  2018.

\bibitem{walker}
Robert~J. Walker.
\newblock {\em Algebraic Curves}.
\newblock Springer-Verlag, New York-Heidelberg, 1978.
\newblock Reprint of the 1950 edition.

\end{thebibliography}

\end{document}